\documentclass[11pt]{amsart}
\usepackage[T1]{fontenc}
\usepackage[utf8]{inputenc}
\usepackage{mathpazo}
\usepackage{amsmath,amssymb,amsthm,mathtools}
\usepackage[a4paper,margin=28mm]{geometry}
\usepackage{microtype,booktabs,array,enumitem,needspace}
\usepackage[dvipsnames]{xcolor}
\usepackage[colorlinks=true,linkcolor=MidnightBlue,citecolor=MidnightBlue,urlcolor=MidnightBlue]{hyperref}
\hypersetup{pdftitle={Regular bundles on orbifolds II: ranks and obstructions},pdfauthor={Enrique Becerra and Ernesto Lupercio}}
\setlist[enumerate]{label=\textup{(\roman*)},leftmargin=*,itemsep=3pt}
\numberwithin{equation}{section}
\newtheorem{theorem}{Theorem}[section]
\newtheorem{proposition}[theorem]{Proposition}
\newtheorem{lemma}[theorem]{Lemma}
\newtheorem{corollary}[theorem]{Corollary}
\newtheorem{maintheorem}{Theorem}

\theoremstyle{definition}
\newtheorem{definition}[theorem]{Definition}
\newtheorem{example}[theorem]{Example}
\newtheorem{question}[theorem]{Question}
\theoremstyle{remark}
\newtheorem{remark}[theorem]{Remark}
\newcommand{\C}{\mathbb C}
\newcommand{\R}{\mathbb R}
\newcommand{\Z}{\mathbb Z}

\newcommand{\G}{\mathcal G}

\newcommand{\rk}{\operatorname{rk}}
\newcommand{\id}{\operatorname{id}}
\newcommand{\lcm}{\operatorname{lcm}}
\newcommand{\ind}{\operatorname{ind}}
\newcommand{\per}{\operatorname{per}}
\newcommand{\smin}{\operatorname{r}_{\min}}
\newcommand{\ord}{\operatorname{ord}}
\newcommand{\eps}{\varepsilon}
\newcommand{\Pic}{\operatorname{Pic}}

\newcommand{\res}{\operatorname{res}}
\newcommand{\Rsharp}{R^{\sharp}}
\newcommand{\floor}[1]{\left\lfloor#1\right\rfloor}
\newcommand{\ceil}[1]{\left\lceil#1\right\rceil}
\newcommand{\tK}{\widetilde K}
\title[Regular bundles II: ranks and obstructions]{Regular bundles on orbifolds II:\\ranks and obstructions}
\author{Enrique Becerra}
\address{Departamento de Matem\'aticas, CINVESTAV, Mexico City, Mexico}
\email{ebecerra@math.cinvestav.mx}
\author{Ernesto Lupercio}
\address{Departamento de Matem\'aticas, CINVESTAV, Mexico City, Mexico}
\email{lupercio@math.cinvestav.mx}
\date{}
\subjclass[2020]{Primary 57R18; Secondary 19L20, 19L47, 19L50, 55R91}
\keywords{Orbifold, orbispace, regular bundle, Adams operation, twisted K-theory, period--index problem}
\begin{document}
\begin{abstract}
A regular complex bundle on an orbispace has fibres that are positive multiples of the stabilizers' regular representations. For stabilizers of order at most two, the sharp uniform rank bound in dimension five is $16$ for paracompact orbispaces and $4$ for closed oriented smooth orbifolds. Rank $16$ occurs with dense free locus and minimum regular rank $4$ on the nonfree locus. We also prove the sharp bound $16$ for closed oriented smooth six-orbifolds. We compute the stable regular index of cyclic gerbes over finite connected CW complexes as the least common multiple of the indices of their character twists. For paracompact orbispaces, optimal integral Adams multipliers give explicit rank bounds in terms of dimension and stabilizer orders. The least positive rank of an integral exterior-power expression with regular fibre and zero reduced class for every regular rank-four bundle on $S^4\times B\mu_2$ is $32$.
\end{abstract}
\maketitle

\section{Introduction}\label{sec:intro}

Let $X$ be an orbispace with finite stabilizer groups $G_x$. A complex vector bundle $E$ on $X$ is \emph{regular} if
\[
 E_x\cong\C[G_x]^{\oplus\rk(E)/|G_x|}
\]
at every point. We require its rank to be positive and constant. How small can that rank be, and how does it depend on dimension, isotropy, and the topology of $X$? We call the locus of points with nontrivial stabilizer the \emph{nonfree locus}.

For stabilizers of order at most two, the optimal uniform bound for the rank of a regular bundle in dimension five is $16$ for paracompact orbispaces (Theorem~\ref{thm:order_two}) and $4$ for closed oriented smooth orbifolds (Theorem~\ref{thm:oriented_five}). The example requiring rank $16$ has a dense free locus, and its entire nonfree locus has minimum regular rank $4$. The attachment of free cells therefore contributes a further obstruction.

For regular rank-four bundles on $S^4\times B\mu_2$, consider a single integral expression in exterior powers whose class has regular fibre and zero reduced part for every such bundle. We prove that the least positive rank of such an expression is $32$ (Theorem~\ref{thm:nonlinearintro}).

In Part~I \cite{BL}, we constructed regular bundles of explicit rank on smooth orbifolds with bounded dimension and stabilizer orders, including noncompact and ineffective orbifolds. The unitary frames of these bundles give quotient presentations. Here we study how small these ranks can be. We determine the optimal integral Adams multipliers in the cellular construction and extend the resulting bounds to paracompact orbispaces through Pardon's approximation theorem. The stable regular index relates regular bundles to the topological period--index problem and yields sharper bounds for cyclic gerbes.

The rank problem has a classical analogue in the immersion problem for manifolds. In positive codimension, the Smale--Hirsch theorem reduces the existence of an immersion into Euclidean space to a bundle monomorphism from the tangent bundle into a trivial bundle \cite{HirschImmersions}. Adem and Gitler used secondary characteristic classes to obtain nonimmersion results for projective spaces \cite{AdemGitler}. Here we seek global $K$-classes with prescribed regular fibre representations. The connectivity of complex Stiefel manifolds then controls the rank needed to represent these classes by actual bundles.

For smooth orbifolds, regular bundles also give embeddings into linear quotient stacks. If $E$ is a smooth Hermitian regular bundle of rank $r$, its unitary frame manifold $P$ has finitely many $U(r)$-orbit types. The equivariant embedding theory of Mostow, Palais, and Wasserman \cite{Mostow,Palais,Wasserman} gives an embedding $P\hookrightarrow V$ into a finite-dimensional real representation, hence a representable smooth embedding
\[
 X\simeq[P/U(r)]\hookrightarrow[V/U(r)].
\]
The standard associated complex bundle of rank $r$ restricts to $E$. Thus the minimum regular rank is the least $r$ for which such an embedding exists and the pullback of the standard bundle is regular; $V$ may vary (Proposition~\ref{prop:linear_embedding}).

\subsection{The rank bound}
We use Pardon's notion of an orbispace \cite{Pardon}: a separated topological stack locally presented as the quotient of a Hausdorff space by a finite group. Paracompactness and dimension refer to its coarse space; dimension at most $n$ means that every open cover has a locally finite refinement whose nerve has dimension at most $n$. Suppose that $\dim|X|\leq n$ and that the stabilizer orders are bounded. Put
\begin{equation}\label{eq:invariants}
 L_o=\lcm_{x\in X}|G_x|,\qquad
 L_e=\lcm_{x\in X}\exp(G_x),\qquad
 t=t(n)=\max\{0,\floor{(n-1)/2}\}.
\end{equation}
For a positive integer $a$ and $L\geq2$, set
\[
 a_L=\prod_{p\mid L}p^{v_p(a)},\qquad
 \delta_L=\begin{cases}1,&v_2(L)=1,\\0,&v_2(L)\ne1.\end{cases}
\]
Define
\begin{equation}\label{eq:rank}
 \Rsharp(n;L_o,L_e)=
 L_oL_e^t(t!)_{L_e}2^{\delta_{L_e}\floor{t/2}}
 \quad(L_e\geq2),
\end{equation}
and set $\Rsharp(n;1,1)=1$.

\begin{maintheorem}\label{thm:main}
Every paracompact orbispace satisfying the preceding hypotheses has a regular complex bundle of rank $\Rsharp(n;L_o,L_e)$. If $X$ is a Hausdorff second-countable smooth orbifold without boundary, the bundle can be chosen smooth and Hermitian. Its unitary frame bundle gives a presentation
\[
 X\simeq[M/U(\Rsharp(n;L_o,L_e))]
\]
by a smooth manifold with a proper locally free action. The manifold $M$ is compact if and only if $|X|$ is compact.
\end{maintheorem}

If only $|G_x|\leq B$ is specified, one can use $L_o=L_e=L=\lcm(1,\ldots,B)$ in the formula. These numbers need only be common multiples of the actual orders and exponents. For fixed $L\geq2$, put
\begin{equation}\label{eq:constant}
 C(L)=L\,2^{\delta_L/2}\prod_{p\mid L}p^{1/(p-1)}.
\end{equation}
Then $\Rsharp(n;L_o,L_e)\leq L_oC(L_e)^t$, an exponential bound in dimension for fixed isotropy.

At the $(2s+1)$-cell attachment, the proof replaces the multiplier $(1+L)^s-1$ of Part~I by
\begin{equation}\label{eq:gdef}
 g_s(L)=\gcd\{k^s-1:k\geq1,\ k\equiv1\pmod L\}.
\end{equation}
We set $g_s(1)=1$ for trivial isotropy. For $L\geq2$, an integral combination of at most three Adams operations attains $g_s(L)$ and cancels the reduced boundary class whenever the stabilizer exponent divides $L$. We compute this multiplier, including its exceptional $2$-primary factor, and realize the resulting virtual classes by actual bundles. The construction works on locally finite simplicial complexes of groups. Pardon's representable approximation theorem \cite[Proposition 4.9]{Pardon} then gives Theorem~\ref{thm:main}.

\subsection{Operation bounds and geometric obstructions}
Allowing positive Adams indices outside the congruence class in \eqref{eq:gdef} gives the same least multiplier when all groups of order at most $B$ are included.

\begin{maintheorem}\label{thm:additiveintro}
Fix $B\geq2$ and $s\geq1$, and let $L=\lcm(1,\ldots,B)$. Suppose a finite integral sum $\Theta=\sum_kn_k\psi^k$, with $k\geq1$, kills every reduced class $\beta_s a$ on $S^{2s}\times BG$ for every group $G$ of order at most $B$. If $\Theta(\C[G])=c\C[G]$ for all these groups, then $g_s(L)\mid c$. Equality is attainable.
\end{maintheorem}

Here $\beta_s$ generates $\tK^0(S^{2s})$, and $a$ ranges over $R(G)$. A smaller family of groups can admit a smaller multiplier, as Example~\ref{ex:q8} shows.

We next consider integral exterior-power expressions on regular rank-four bundles over $S^4\times B\mu_2$. An \emph{integral exterior-power expression in $E$} means a polynomial with integer coefficients in $[E],\lambda^2[E],\ldots,\lambda^{\rk E}[E]$, with the integer coefficients represented by trivial virtual bundles.

\begin{maintheorem}\label{thm:nonlinearintro}
Let $\eps$ be the nontrivial character of $\mu_2$, and consider rank-four bundles on $S^4\times B\mu_2$ with classes
\[
 [E]=2+2\eps+\beta_2(u+v\eps),\qquad u,v\in\Z.
\]
If one integral exterior-power expression has regular fibre and zero reduced class for every pair $(u,v)$, its rank is divisible by $32$. In particular, a positive rank is at least $32$. Rank $32$ is attained by $6E^{\otimes2}-E^{\otimes3}$.
\end{maintheorem}

For a finite connected complex of groups, we define the \emph{stable regular index} $I_K(X)$ using the compatible family $\omega_x=(L_o/|G_x|)\C[G_x]$. It is the least positive $a$ for which $a\omega$ is the fibre family of a virtual bundle. We prove
\[
 g_s(L_e)d_{2s+1}=0\quad\text{on }E_{2s+1}^{0,0}
\]
in the cellular Atiyah--Hirzebruch spectral sequence. The index is the product of the successive orders of these differentials on the multiples of $\omega$ that survive. The Stiefel connectivity estimate in Corollary~\ref{cor:positive} gives a sufficient condition for a virtual class to be represented by an actual bundle.

For a banded $\mu_e$-gerbe over a finite connected CW complex $M$, let $\xi\in H^2(M;\Z/e)$ be its class and $\alpha=\beta_e(\xi)$. The weight decomposition gives
\[
 I_K(\G)=\lcm_{0\leq j<e}\ind(j\alpha).
\]
Here $\ind(j\alpha)$ is the gcd of the ranks of finite-rank $j\alpha$-twisted bundles. Antieau--Williams' period--index results \cite{AW2014,AW6,AW2021} give smaller ranks on cyclic gerbes and show that minimum regular ranks can grow without bound with a fixed cyclic stabilizer group as dimension increases.

\subsection{Sharp bounds with varying isotropy}
Every complex of groups with stabilizers $1$ or $\mu_2$ admits a representable map to a banded $\mu_2$-gerbe over the same coarse complex. The construction retains the twisting elements and uses the unique injections of the cell groups into $\mu_2$. Gerbe rank bounds therefore pull back to these complexes and, through Pardon's approximation, to all paracompact orbispaces with stabilizers of order at most two.

\begin{maintheorem}\label{thm:order_two}
Every paracompact orbispace of dimension at most $n$ with stabilizers of order at most two admits a regular bundle of rank
\begin{equation}\label{eq:order_two_rank}
 R_2(n)=2^{\,t+1+v_2(t!)},\qquad t=t(n).
\end{equation}
The bound $R_2(5)=16$ is sharp. It is attained by a finite connected five-dimensional complex of groups with dense free locus and with minimum regular rank four on its entire nonfree locus.
\end{maintheorem}

To attain the bound, start with a suitable gerbe on a finite five-complex and make the stabilizers on the open five-simplices trivial. In rank and odd-multiplicity coordinates, the representation-ring cochains split as
\[
 C^*(M;\Z)\oplus C^*(S;\Z),
\]
where $M$ is the coarse complex and $S$ is its nonfree subcomplex. The comparison map preserves a value of $d_5$ of order four in the coarse cohomology summand. This class restricts to zero on $S$. Together with the first obstruction of order two, it forces regular rank sixteen.

For a closed oriented smooth five-orbifold with these stabilizers, each connected component $M$ of its coarse space satisfies $H^5(M;\Z)\cong\Z$. This reduces the bound to four.

\begin{maintheorem}\label{thm:oriented_five}
Every closed oriented smooth five-orbifold with stabilizers of order at most two admits a regular bundle of rank four. This bound is attained by a connected effective example with both stabilizer types and a dense free locus.
\end{maintheorem}

The example is an explicit quotient of $S^3\times S^2\times S^1$ by $O(2)$. A four-dimensional representation gives a regular bundle. A mapping-torus suborbifold rules out rank two by the determinant-line obstruction. The mapping torus and its product with a circle attain the sharp rank-four bounds in dimensions three and four.

In dimension six, the bound $R_2(6)=16$ is attained in the closed oriented smooth category. Corollary~\ref{cor:sixsharp} derives this from Crowley--Grant's period-two, index-eight class on an oriented six-manifold \cite[Theorem 1.4(2)]{CG}. The resulting orbifold is a banded $\mu_2$-gerbe, with ineffective stabilizer at every point.

\subsection{Relation to earlier work}
Henriques--Metzler study presentations of noneffective orbifolds through frame constructions and gerbes \cite{HM}. Kali\v{s}nik constructs faithful representations of orbifold groupoids on continuous families of finite-dimensional Hilbert spaces \cite{Kalisnik}. L\"uck--Oliver construct regular equivariant bundles for proper actions of discrete groups \cite[Corollary 2.7]{LO}. Pardon proves existence of regular bundles for the orbispaces considered here \cite[Theorem 1.1]{Pardon}. His approximation theorem extends our cellular construction to paracompact orbispaces, and his $K$-theory supplies the cellular spectral sequence used for the stable regular index.

For the relation between gerbes, twisted bundles, and twisted $K$-theory, see Lupercio--Uribe \cite{LU}, Atiyah--Segal \cite{AS}, and Tu--Xu--Laurent-Gengoux \cite{TXLG}.

Sections~\ref{sec:cells}--\ref{sec:operations} give the cellular construction and the optimality results for universal operations. Sections~\ref{sec:ahss}--\ref{sec:examples} study the stable regular index, its relation to gerbes, and the examples attaining the geometric bounds.

\section{Cellular bundles and realization}\label{sec:cells}

Let $X=\|(Z,G)\|$ be the realization of a locally finite finite-dimensional simplicial complex of finite groups in the sense of \cite[Definition 2.1]{Pardon}. Write $X^j$ for its $j$-skeleton. A characteristic map has the form
\[
 (D^j,\partial D^j)\times BG_\sigma\longrightarrow(X^j,X^{j-1}).
\]
It is representable. For a face $\sigma\subset\tau$, the group map is an injection $G_\tau\hookrightarrow G_\sigma$, with the specified conjugacies from the complex of groups.

\begin{lemma}[Cellular clutching]\label{lem:clutch}
Bundles on $X^j$ are equivalent to the data of a bundle on $X^{j-1}$, a $G_\sigma$-bundle on each characteristic $j$-disk, and boundary isomorphisms to the pullback of the bundle on $X^{j-1}$. Compatible bundle maps glue. The equivalence commutes with direct sums, tensor products, and exterior powers.
\end{lemma}
\begin{proof}
Use the open-star charts of the complex of groups, restricted to $X^j$. In the chart over $\operatorname{St}(\sigma)$, with group $Q=G_\sigma$, the portion of a characteristic simplex $\tau\supset\sigma$ pulls back to
\[
 Q\times_{G_\tau}\bigl(\tau\cap\operatorname{St}(\sigma)\bigr).
\]
The subgroup injection and the face identifications are those of the complex of groups, including its twisting elements; see the description following \cite[Definition 2.1]{Pardon}.

Induce each prescribed disk bundle to $Q$ and identify its boundary with the pullback of the bundle on $X^{j-1}$. A collar of the disk boundary and radial retraction extend the boundary isomorphism to bundle charts near the attaching locus. Restrict these charts to the portions under consideration. Together with the disk-interior charts, they give a $Q$-equivariant bundle on the star. Local finiteness reduces this construction to finitely many incident cells near each point.

Restriction to $X^{j-1}$ and the characteristic pieces is fully faithful. Indeed, compatible bundle maps determine a unique fibrewise linear map on the attached bundle. In local bundle trivializations its matrix entries are continuous on each characteristic piece and agree under the attaching identifications. The cellular adjunction topology makes them continuous on the chart. The same argument applies after restriction to an open subset. Equivariance and equality of maps can likewise be checked on these pieces.

On an overlap of two star charts, pass to common chart refinements over $X^j$. After applying the subgroup restrictions and conjugacies of the chart maps, the two bundles induce the same data on $X^{j-1}$ and on each characteristic piece. Full faithfulness gives the unique comparison isomorphism inducing these identifications. Such comparisons agree on further refinements and hence define an isomorphism on the overlap. On a triple overlap, the composite and the direct comparison agree on every piece by the coherence identity of the complex of groups; full faithfulness makes them equal. Open descent therefore gives a bundle on $X^j$.

Restriction recovers the prescribed data. Conversely, the piecewise identity for an existing bundle glues to an isomorphism with its reconstruction. These maps show that restriction and clutching are quasi-inverse. The same construction glues bundle maps and commutes with the stated fibrewise operations.
\end{proof}

This is the topological cellular form of \cite[Theorem 2.1]{BL}, using the local charts of the complex of groups. Finite-group averaging and partitions of unity give invariant Hermitian metrics; see also \cite[Lemma 5.1]{Pardon}.

\begin{lemma}[Relative embeddings]\label{lem:embedding}
Let $Y\subset X$ be a subcomplex and let $D,A$ be complex bundles on $X$, with a specified monomorphism $D|_Y\to A|_Y$. On a characteristic $j$-cell outside $Y$, let $d_\rho,a_\rho$ denote the multiplicities of an irreducible representation $\rho$ in $D,A$. If
\begin{equation}\label{eq:gap}
 a_\rho\geq d_\rho,\qquad
 2(a_\rho-d_\rho)\geq j-1\quad\text{whenever }d_\rho>0,
\end{equation}
the monomorphism extends over $X$.
\end{lemma}
\begin{proof}
A constant-group disk admits isotypic decomposition, and its multiplicity bundles are trivial. The space of injective equivariant maps is homotopy equivalent to the product of complex Stiefel manifolds $V_{d_\rho}(\C^{a_\rho})$ over the nonzero $d_\rho$. The sphere fibrations that forget the last frame vector show that this Stiefel manifold is $2(a_\rho-d_\rho)$-connected. Thus the boundary map on $S^{j-1}$ extends under \eqref{eq:gap}. At a vertex only nonemptiness is required. Use the full space of injective maps to keep the prescribed boundary value fixed. Lemma~\ref{lem:clutch} glues these extensions, dimension by dimension. An invariant metric supplies an orthogonal quotient bundle.
\end{proof}

Let $K^0_{\mathrm{vb}}(X)$ denote the Grothendieck group of bounded-rank complex bundles. The construction below takes place in this group, using finite expressions in bundles. Section~\ref{sec:ahss} introduces the associated cohomology theory.

\begin{corollary}[Positive virtual classes]\label{cor:positive}
If $x=[A]-[D]\in K^0_{\mathrm{vb}}(X)$ has nonnegative multiplicities $u_\rho$ and $2u_\rho\geq n-1$ on every cell of an $n$-dimensional complex, then $x$ is represented by an actual bundle.
\end{corollary}
\begin{proof}
Apply Lemma~\ref{lem:embedding} with $Y=\varnothing$ and $a_\rho-d_\rho=u_\rho$, and take $A/D$.
\end{proof}

\begin{lemma}[Sphere cancellation]\label{lem:cancellation}
A complex bundle of rank $u$ on $S^q$ with trivial reduced $K$-class is trivial if $2u\geq q$. The same statement holds on $S^q\times BG$ if the inequality holds for every nonzero isotypic multiplicity.
\end{lemma}
\begin{proof}
The sphere fibrations $U(u)\to U(u+1)\to S^{2u+1}$ show that $\pi_q(BU(u))\to\pi_q(BU)$ is injective for $q\leq2u$; see \cite{Hatcher}. A stable class therefore determines the bundle in this range. Apply this to each multiplicity bundle. A zero multiplicity contributes the zero bundle.
\end{proof}

\begin{lemma}[The first two skeleta]\label{lem:initial}
If every cell group has order dividing $L_o$, there is a regular bundle of rank $L_o$ on $X^2$.
\end{lemma}
\begin{proof}
At a vertex with group $H$, take $(L_o/|H|)\C[H]$. If $J\hookrightarrow H$, then
\begin{equation}\label{eq:restriction}
 \res^H_J\C[H]\cong\C[J]^{\oplus[H:J]}.
\end{equation}
The two endpoint representations of a characteristic edge are consequently isomorphic to $(L_o/|J|)\C[J]$, so they admit clutching to a constant disk bundle. Over the boundary of a two-cell, the isotypic multiplicity bundles are complex bundles on $S^1$ and are trivial. They extend across the disk. Lemma~\ref{lem:clutch} glues these extensions.
\end{proof}

\section{The least integral multiplier}\label{sec:arithmetic}

For a finite group $J$, equivariant isotypic decomposition gives
\[
 K^0(S^{2s}\times BJ)=K^0(S^{2s})\otimes R(J).
\]
Write $K^0(S^{2s})=\Z\oplus\Z\beta_s$, with $\beta_s^2=0$. Adams operations satisfy
\begin{equation}\label{eq:adams}
 \psi^k(\beta_s)=k^s\beta_s,
 \qquad \chi_{\psi^k(a)}(g)=\chi_a(g^k).
\end{equation}
The first formula follows from Bott periodicity, or the Chern character and torsion-freeness of $K^0(S^{2s})$; the second is the character formula for Adams operations. These standard operations are integral polynomials in exterior powers \cite{Hatcher}. In particular, $\psi^k=\id$ on $R(J)$ when $k\equiv1\pmod L$ and $\exp(J)\mid L$.

\begin{proposition}\label{prop:least}
For $L\geq2$ and $s\geq1$, the least positive integer $c$ for which there are finitely many integers $n_i$ and positive $k_i\equiv1\pmod L$ satisfying
\begin{equation}\label{eq:moments}
 \sum_i n_i=c,\qquad \sum_i n_i k_i^s=0
\end{equation}
is $g_s(L)$ from \eqref{eq:gdef}.
\end{proposition}
\begin{proof}
The equations give $c=-\sum_i n_i(k_i^s-1)$, so $g_s(L)\mid c$. Conversely, a greatest common divisor is an integral linear combination of finitely many of the integers in its defining set. Write $g_s(L)=\sum_i b_i(k_i^s-1)$. Then
\[
 \Theta_s=\left(g_s(L)+\sum_i b_i\right)\id-\sum_i b_i\psi^{k_i}
\]
has coefficient sum $g_s(L)$ and weighted coefficient sum zero.
\end{proof}

\begin{proposition}\label{prop:valuations}
For every prime $p$,
\begin{equation}\label{eq:valuations}
 v_p(g_s(L))=
 \begin{cases}
 v_p(L)+v_p(s),&p\mid L,\ p\text{ odd},\\
 v_2(L)+v_2(s),&p=2,\ 4\mid L,\\
 1,&p=2,\ 2\parallel L,\ s\text{ odd},\\
 2+v_2(s),&p=2,\ 2\parallel L,\ s\text{ even},\\
 0,&p\nmid L.
 \end{cases}
\end{equation}
Consequently
\begin{equation}\label{eq:gclosed}
 g_s(L)=L\,s_L\,2^{\delta_L\mathbf1_{2\mid s}},\qquad
 \prod_{s=1}^t g_s(L)=L^t(t!)_L2^{\delta_L\floor{t/2}}.
\end{equation}
In particular $2\leq g_s(L)\leq2Ls$.
\end{proposition}
\begin{proof}
We recall the elementary valuation identities needed here. If $p$ is odd and $p\mid a-1$, then
\[
 v_p(a^s-1)=v_p(a-1)+v_p(s).
\]
To see this, an exponent prime to $p$ preserves the valuation by expanding $(1+b)^r-1$ with $p\mid b$; raising to the $p$th power increases it by exactly one. Iteration proves the formula. The same argument works for $p=2$ when $4\mid a-1$. For odd $a$, an odd exponent preserves $v_2(a-1)$, and an even exponent gives
\begin{equation}\label{eq:lte2}
 v_2(a^s-1)=v_2(a-1)+v_2(a+1)+v_2(s)-1.
\end{equation}
Indeed, factor off the first square and then apply the $4\mid a^2-1$ case to the remaining powers of two.

For odd $p\mid L$ and $k\equiv1\pmod L$ with $k>1$, we therefore have $v_p(k^s-1)\geq v_p(L)+v_p(s)$, with equality at $k=1+L$. If $4\mid L$, then $v_2(k+1)=1$, and the same conclusion holds at two.

Suppose $2\parallel L$. For odd $s$, the minimum is one. For even $s$, the first two terms on the right of \eqref{eq:lte2} sum to at least three. The progression $1+L\Z$ contains positive integers congruent to $3$ or $5$ modulo $8$, so the minimum is $2+v_2(s)$. Finally, when $p\nmid L$, the Chinese remainder theorem gives a positive $k$ with $k\equiv1\pmod L$ and $k\equiv0\pmod p$. Then $k^s-1$ is a $p$-adic unit. Taking the minimum valuation at each prime proves \eqref{eq:valuations}. Multiplication gives \eqref{eq:gclosed} and the stated estimates.
\end{proof}

\begin{proposition}[Three Adams operations]\label{prop:three}
The least multiplier in Proposition~\ref{prop:least} is realized by an integral linear combination of $\id$, $\psi^{k_1}$, and $\psi^{k_2}$.
\end{proposition}
\begin{proof}
Let $k_1=1+L$ and $A=k_1^s-1$. For each $p\mid L$, impose a congruence on $k_2$ attaining the minimum in \eqref{eq:valuations}. It can be imposed compatibly with $k_2\equiv1\pmod{p^{v_p(L)}}$. For every prime $\ell\mid A$ not dividing $L$, impose $k_2\equiv0\pmod\ell$. The Chinese remainder theorem supplies a positive $k_2\equiv1\pmod L$ with these properties. If $D=k_2^s-1$, then $\gcd(A,D)=g_s(L)$. Choose $uA+vD=g_s(L)$ and take
\begin{equation}\label{eq:three}
 \Theta_s=(g_s(L)+u+v)\id-u\psi^{k_1}-v\psi^{k_2}.
\end{equation}
\end{proof}

For instance, $2=-19(3^3-1)+4(5^3-1)$ gives
\[
 \Theta_3=-13\id+19\psi^3-4\psi^5\qquad(L=2).
\]
For any bundle, separating the positive and negative terms of the corresponding exterior-power expression gives a virtual bundle $[A]-[D]$.

\begin{remark}[A finite gcd formula]\label{rem:certificate}
The gcd is determined by the first $s+1$ values:
\begin{equation}\label{eq:finitegcd}
 g_s(L)=\gcd_{0\leq j\leq s}\big((1+Lj)^s-1\big).
\end{equation}
The polynomial $P(j)=(1+Lj)^s-1$ has degree $s$. Its Newton expansion in the integer-valued polynomials $\binom{j}{a}$ shows that $P(j)$, for every nonnegative integer $j$, is an integral linear combination of $P(0),\ldots,P(s)$. This proves \eqref{eq:finitegcd}.
\end{remark}

\section{Construction of regular bundles}\label{sec:construction}

We first prove Theorem~\ref{thm:main} for a locally finite complex of groups of dimension at most $n$. We allow $L_o,L_e$ to be any common multiples of the cell-group orders and exponents, respectively, with $L_e\geq2$. Set
\begin{equation}\label{eq:rstage}
 r_0=L_o,\qquad r_s=L_o\prod_{i=1}^s g_i(L_e).
\end{equation}

\begin{proposition}\label{prop:induction}
There is a regular bundle of rank $r_s$ on $X^{2s+2}$, whenever this skeleton occurs. If $\dim X=2s+1$, the same rank suffices on $X$.
\end{proposition}
\begin{proof}
Lemma~\ref{lem:initial} gives $s=0$. Suppose $s\geq1$ and $E$ is the regular bundle of rank $r_{s-1}$ on $X^{2s}$. On the boundary of a characteristic $(2s+1)$-cell with group $J$, write
\[
 [E]=W+\beta_s a,\qquad
 W=\frac{r_{s-1}}{|J|}\C[J],\quad a\in R(J).
\]
Choose $\Theta_s$ by Proposition~\ref{prop:least}, with $L=L_e$. Equations \eqref{eq:adams} and \eqref{eq:moments} give
\begin{equation}\label{eq:cancel}
 \Theta_s[E]|_{S^{2s}\times BJ}=g_s(L_e)W.
\end{equation}

Newton identities express each positive Adams operation as an integral polynomial in exterior powers. Collect positive and negative terms in the finite expression $\Theta_s[E]$ to write it as $[A]-[D]$ for actual bundles on $X^{2s}$. At every point its fibre representation is $g_s(L_e)E_x$. Thus its multiplicity in an irreducible $\rho$ of the stabilizer $H$ is
\begin{equation}\label{eq:multiplicity}
 u_\rho=\frac{r_s}{|H|}\dim\rho
 \geq\prod_{i=1}^s g_i(L_e)\geq2^s.
\end{equation}
The estimate bounds the multiplicities of the virtual difference $[A]-[D]$, as required by Lemma~\ref{lem:embedding}.

Since $2^{s+1}>2s+1$, Lemma~\ref{lem:embedding} embeds $D$ into $A$ on $X^{2s}$. Let $F=A/D$. On each characteristic boundary its multiplicity bundles have trivial reduced classes by \eqref{eq:cancel}. Their ranks satisfy $2u_\rho>2s$, so Lemma~\ref{lem:cancellation} makes them trivial. Clutch $F$ to the constant bundle with fibre $g_s(L_e)W$ on the disk. This extends $F$ over $X^{2s+1}$.

If $X^{2s+2}$ occurs, its characteristic boundaries are $S^{2s+1}\times BJ$. The reduced complex $K^0$ of an odd sphere vanishes. The bound $2u_\rho>2s+1$ therefore shows that the multiplicity bundles on these boundaries are trivial and permits the next attachment with the same rank. Lemma~\ref{lem:clutch} completes the induction.
\end{proof}

The proof uses one finite bundle expression in each dimension. Local finiteness suffices for clutching, and the multiplicity estimate \eqref{eq:multiplicity} is independent of the number of cells.

\begin{proof}[Proof of Theorem~\ref{thm:main}]
If all stabilizers are trivial, the trivial line bundle suffices. Otherwise, Pardon's approximation theorem \cite[Proposition 4.9]{Pardon} supplies a representable map
\[
 f:X\longrightarrow\|(Z,G)\|
\]
to a locally finite complex of finite groups of dimension at most $n$. Its vertex groups can be taken to be stabilizer groups of $X$. Every other cell group injects into a vertex group. Its order consequently divides $L_o$ and its exponent divides $L_e$. Proposition~\ref{prop:induction} constructs a bundle of rank $r_t$ on this complex. Pullback is regular by \eqref{eq:restriction} and representability. Formula \eqref{eq:gclosed} identifies $r_t$ with \eqref{eq:rank}.

For a smooth orbifold, \cite[Proposition 4.1]{BL} gives a smooth Hermitian bundle $E^\infty$ and a continuous equivariant isomorphism $\Phi:E\to E^\infty$. The local smoothing construction is as follows. On a sufficiently small finite-group chart, averaging a local trivialization gives a continuous equivariant identification with the product of the chart and the representation at its centre. This equips the bundle locally with a smooth structure. On an overlap, a continuous equivariant bundle isomorphism $h$ between two such structures can be approximated by a smooth equivariant isomorphism $k$: approximate its matrix entries, average over the finite group, and keep the error small enough to preserve invertibility.

For the gluing over open substacks $A$ and $B$, choose the error so that $kh^{-1}$ tends to the identity at the boundary of $A\cap B$ in $B$. It then extends by the identity to a continuous equivariant automorphism on $B$. Transporting the local smooth structure by this automorphism makes it agree with the given structure on $A$. A countable cover and successive gluings produce a global smooth structure; every point has a neighbourhood on which the structure is fixed after finitely many steps. This is the overlap construction in the proof of \cite[Proposition 4.1]{BL}. A smooth invariant Hermitian metric is obtained by averaging and a partition of unity. For each $g\in G_x$, equivariance gives
\[
 \Phi_x\,\rho_x(g)=\rho_x^\infty(g)\,\Phi_x.
\]
Thus every full-stabilizer representation, including its ineffective part, is unchanged.

Regular fibres are faithful. Hence the presenting \'etale groupoid acts freely and properly on the unitary frame atlas, and its quotient is a smooth manifold. The commuting right $U(\rk E)$-action has stabilizers conjugate to the original $G_x$ and gives the asserted presentation; see \cite[Section 5]{BL}. This action is proper because the group is compact. Its orbit map is proper, so the frame manifold is compact exactly when its orbit space $|X|$ is compact.
\end{proof}

\begin{corollary}[Growth]\label{cor:growth}
For fixed $L\geq2$,
\[
 \prod_{s=1}^t g_s(L)\leq C(L)^t,
 \qquad
 \lim_{t\to\infty}\left(\prod_{s=1}^t g_s(L)\right)^{1/t}=C(L).
\]
\end{corollary}
\begin{proof}
Legendre's formula gives
\[
 v_p(t!)=\sum_{a\geq1}\floor{t/p^a}\leq\frac{t}{p-1},
 \qquad \frac{v_p(t!)}t\longrightarrow\frac1{p-1}.
\]
Apply these formulas to \eqref{eq:gclosed}, and use $\floor{t/2}/t\to1/2$.
\end{proof}

Table~\ref{tab:ranks} compares the Adams construction with Part~I when $L_o=L_e=\lcm(1,\ldots,B)$. The two rank bounds agree through dimension six when $B=2$; their first difference occurs at $s=3$, where $g_3(2)=2$ instead of $3^3-1=26$. The geometric bound of Theorem~\ref{thm:order_two} improves the order-two ranks already in dimension five; its values appear in Table~\ref{tab:order_two}.

\begin{table}[ht]
\centering
\caption{Regular ranks from Part~I and Theorem~\ref{thm:main} in dimension at most $n$, for stabilizer orders at most $B$ and $L=L_o=L_e=\lcm(1,\ldots,B)$.}\label{tab:ranks}
\begin{tabular}{@{}rrrrr@{}}
\toprule
$B$ & $n$ & $L$ & Part~I & $\Rsharp$\\
\midrule
2&3,4&2&4&4\\
2&5,6&2&32&32\\
2&7,8&2&832&64\\
2&9,10&2&66560&1024\\
3&3,4&6&36&36\\
3&5,6&6&1728&864\\
3&7,8&6&590976&15552\\
\bottomrule
\end{tabular}
\end{table}

\begin{corollary}[Quantitative cofinality]\label{cor:cofinal}
Let $f:X\to Y$ be a representable map of paracompact orbispaces, with $\dim|X|\leq d$. Let $F$ be a regular bundle on $Y$, and let $E$ have rank at most $r$ on $X$. There is a bundle monomorphism
\[
 E\longrightarrow f^*F^{\oplus(d+1)r}.
\]
\end{corollary}
\begin{proof}
Every irreducible multiplicity of $f^*F$ is at least one. Every multiplicity of $E$ is at most $r$. Hence $r$ copies of $f^*F$ admit a fibrewise equivariant injection at any chosen point. In a finite-group chart, local trivialization, averaging, and shrinking extend it to a local bundle monomorphism.

Take a locally finite refinement of these neighbourhoods of order at most $d+1$. Its barycentric refinement is a union of $d+1$ families in which distinct members of the same family are disjoint. This is the cover construction in \cite[proof of Corollary 1.4]{Pardon}. Choose a subordinate partition of unity. Multiply each local injection by its partition function, and sum maps from the same family into one copy of $f^*F^{\oplus r}$. At a point at least one partition function is positive, and in that family at most one term is nonzero. The map to the direct sum of all $d+1$ families is therefore injective.
\end{proof}

\begin{proposition}[Embeddings from regular bundles]\label{prop:linear_embedding}
Let $X$ be a Hausdorff second-countable finite-dimensional smooth orbifold without boundary, and let $r\geq1$. Then $X$ admits a smooth regular complex bundle of rank $r$ if and only if there are a finite-dimensional real $U(r)$-representation $V$ and a representable smooth embedding
\[
 j:X\hookrightarrow[V/U(r)]
\]
for which the pullback of the standard associated complex bundle
$[V\times\C^r/U(r)]\to[V/U(r)]$ is regular. Given such a bundle $E$, the embedding may be chosen closed, with this pullback isomorphic to $E$.
\end{proposition}
\begin{proof}
Give $E$ a Hermitian metric. Its unitary frame space $P$ is a manifold and $X\simeq[P/U(r)]$, as in the proof of Theorem~\ref{thm:main}. Every stabilizer $H\subset U(r)$ acts on $\C^r$ as $(r/|H|)\C[H]$, so $|H|$ divides $r$. There are finitely many abstract groups of these orders, and for each one this representation determines its image in $U(r)$ up to conjugacy. Thus $P$ has finitely many orbit types.

For each of these subgroups $H$, the compact homogeneous space $U(r)/H$ admits an equivariant embedding in an orthogonal representation, by \cite[p.~673]{Palais}. Let $W$ be the sum of these representations and the underlying real representation of $\C^r$. Its restriction to each $H$ contains every irreducible real $H$-representation. The slice representations of $P$ therefore embed in finite direct sums of $W$. By \cite[Proposition~1.7]{Wasserman}, $P$ admits local equivariant embeddings in finite direct sums of $W$. For possibly noncompact $P$, Proposition~1.9 and Corollary~1.5 of that paper give a closed equivariant smooth embedding $P\hookrightarrow V$, where $V$ is a finite direct sum of copies of $W$.

Passing to quotient stacks gives $j$. Its base change along $V\to[V/U(r)]$ is the embedding $P\hookrightarrow V$, so $j$ is representable and closed. The standard associated bundle pulls back to $P\times_{U(r)}\C^r\cong E$. The converse follows by pulling back that bundle.
\end{proof}

The normal bundle of $P\hookrightarrow V$ descends to a real bundle $\nu$ on $X$. Splitting the equivariant tangent sequences gives an isomorphism of actual real bundles
\[
 TX\oplus\operatorname{ad}(P)\oplus\nu
 \cong P\times_{U(r)}V,
 \qquad
 \operatorname{ad}(P)=P\times_{U(r)}\mathfrak u(r).
\]
Indeed, $TP\cong\pi^*TX\oplus(P\times\mathfrak u(r))$ and $TP\oplus\nu_P\cong P\times V$, where $\pi:P\to X$ and $\nu_P$ is the normal bundle. Invariant metrics supply both splittings. For an ordinary manifold the minimum regular rank is one; its minimum Euclidean immersion or embedding dimension is a separate invariant.

\section{Optimality of universal operations}\label{sec:operations}

\subsection{Arbitrary positive Adams indices}
We prove Theorem~\ref{thm:additiveintro} in its algebraic form.

\begin{theorem}\label{thm:additive}
Let $L=\lcm(1,\ldots,B)$, where $B\geq2$. Suppose $\Theta=\sum_kn_k\psi^k$ is a finite integral sum with positive indices and, for every group $G$ of order at most $B$,
\begin{equation}\label{eq:universal}
 \sum_kn_kk^s\psi^k(a)=0\quad(a\in R(G)),
 \qquad
 \sum_kn_k\psi^k(\C[G])=c\C[G].
\end{equation}
Then $g_s(L)\mid c$.
\end{theorem}
\begin{proof}
Fix $p\mid L$ and put $q=p^{v_p(L)}$. Since $q\leq B$, we can use the cyclic group $C_q$. Write $z$ for a generating character. The first equation of \eqref{eq:universal}, applied to $z$, says
\begin{equation}\label{eq:residue}
 \sum_{k\equiv a\; (\mathrm{mod}\ q)}n_kk^s=0
 \qquad(a\in\Z/q).
\end{equation}
Let $h=g_s(q)$, a power of $p$. If $a$ is a unit modulo $q$, all integers $k\equiv a\pmod q$ have the same $s$th power modulo $h$. For odd $p$, or for $p=2$ and $q\geq4$, apply the valuation calculation of Proposition~\ref{prop:valuations} to $k/a$ in the $p$-adic integers. Equivalently, use its binomial proof after clearing the unit $a^s$. For $q=2$ and even $s$, every odd $k$ satisfies
\[
 k^s\equiv1\pmod{2^{2+v_2(s)}};
\]
for odd $s$ the required modulus is just two. Thus \eqref{eq:residue} gives
\[
 h\mid\sum_{k\equiv a\; (\mathrm{mod}\ q)}n_k
 \qquad(p\nmid a).
\]

The coefficient of $z$ in $\psi^k(\C[C_q])$ counts the solutions of $kj\equiv1\pmod q$. It is one when $p\nmid k$ and zero otherwise. Comparing this coefficient in the second equation of \eqref{eq:universal} yields
\[
 c=\sum_{p\nmid k}n_k.
\]
Hence $h\mid c$. Multiplying over the primes dividing $L$ gives $g_s(L)\mid c$ by \eqref{eq:valuations}. Proposition~\ref{prop:least} attains equality, proving Theorem~\ref{thm:additiveintro}.
\end{proof}

The next example shows how restricting the stabilizer groups can lower the multiplier.

\begin{example}[A smaller multiplier for quaternion stabilizers]\label{ex:q8}
Every irreducible complex character of $Q_8$ is invariant under inversion. For odd $k$, the power map on each element of $Q_8$ is either identity or inversion. Hence $\psi^k=\id$ on $R(Q_8)$ for every odd $k$. On a constant-$Q_8$ gerbe, the construction can therefore use $g_s(2)$ with initial rank eight. The first step uses $3\id-\psi^3$ and produces rank $16$, while the bound obtained from the group exponent has rank $8\cdot4=32$.

On $R(C_4)$, $\psi^3$ interchanges the two faithful characters. The preceding argument therefore depends on the chosen family of stabilizers.
\end{example}

\subsection{Tensor cancellation}
Tensor products give further cancellation formulas. Their ranks can be compared with those obtained from Adams operations.

\begin{proposition}\label{prop:tensors}
Suppose on $S^{2s}\times BJ$ that $[E]=W+\beta_sa$, where $W$ is regular of rank $r$.
\begin{enumerate}
\item If $s$ is odd, then $[E\otimes\overline E]=W^2$, of rank $r^2$.
\item If $r$ is even, the integral expression
\[
 P_r(E)=\frac{3r}{2}E^{\otimes2}-E^{\otimes3}
\]
has class $(r^2/2)W$, of rank $r^3/2$, for every $s$.
\end{enumerate}
\end{proposition}
\begin{proof}
A regular representation satisfies $Wb=\dim(b)W$ for $b\in R(J)$ and $W^2=rW$. Complex conjugation gives $\overline{\beta_s}=(-1)^s\beta_s$ and $\overline W=W$. Thus
\[
 [E\otimes\overline E]
 =W^2+\beta_sW\big(a+(-1)^s\overline a\big).
\]
For odd $s$ the reduced coefficient vanishes because $\dim a=\dim\overline a$.

Since $\beta_s^2=0$, we also have
\[
 E^2=W^2+2\beta_sWa,\qquad
 E^3=W^3+3\beta_sW^2a.
\]
The reduced term of $P_r(E)$ is $3\beta_s(rW-W^2)a=0$. Its fibre is $(3r/2)W^2-W^3=(r^2/2)W$.
\end{proof}

For $s=1$ the first construction gives rank $L_o^2$ directly. When $L_o=L_e$, this equals the first-stage rank in Theorem~\ref{thm:main}. When the two common multiples differ, the Adams construction gives $L_oL_e$.

\subsection{An exterior-power lower bound}
Theorem~\ref{thm:nonlinearintro} concerns integral polynomials in the exterior powers of the given bundle $E$, with integer coefficients represented by trivial virtual bundles. The universal identities for exterior powers express every finite composition of sums, differences, products, and exterior powers in this form. The lower bound therefore applies to every such composition that satisfies the stated cancellation conditions.

\begin{proof}[Proof of Theorem~\ref{thm:nonlinearintro}]
Put $W=2+2\eps$, $\beta=\beta_2$, and $a=u+v\eps$. Every pair $(u,v)\in\Z^2$ occurs: rank-two complex bundles on $S^4$ realize every element of $\tK^0(S^4)$, since $\pi_3U(2)\to\pi_3U$ is an isomorphism. Choose one such bundle in each weight.

Let $x=[E]$ and $y=\lambda^2[E]$. The exterior powers on these boundary classes satisfy
\begin{equation}\label{eq:lambdas}
 \begin{split}
 x&=W+\beta(u+v\eps),\\
 y&=2+4\eps+2\beta(u+v)\eps,\\
 \lambda^3[E]&=x,\qquad \lambda^4[E]=1.
 \end{split}
\end{equation}
For the second formula use $2\lambda^2x=x^2-\psi^2x$, noting that $\psi^2(\eps)=1$ and $\psi^2(\beta)=4\beta$. The determinant of $E$ is trivial: its stabilizer character is $\eps^2=1$ and $H^2(S^4;\Z)=0$. Hence $\lambda^3E=E^*$. Duality fixes $\beta_2$ and both characters, so $[E^*]=[E]$. Exterior powers above four vanish. Every expression in the theorem therefore reduces to a polynomial $P(x,y)\in\Z[x,y]$ on these classes.

Evaluate characters at the identity and the nontrivial element of $\mu_2$. With $d=u+v$ and $h=u-v$, the two evaluations of $(x,y)$ are
\begin{equation}\label{eq:characterpairs}
 (4+\beta d,\ 6+2\beta d),\qquad
 (\beta h,\ -2-2\beta d).
\end{equation}
A virtual representation of $\mu_2$ is regular exactly when its character at the nontrivial element is zero. The fibre condition gives $P(0,-2)=0$. Vanishing of the reduced coefficient in the second pair, for all integers $u,v$, gives
\[
 P_x(0,-2)=P_y(0,-2)=0.
\]
Indeed, its coefficient is $hP_x(0,-2)-2dP_y(0,-2)$; setting $(u,v)=(1,0)$ and $(0,1)$ proves both equalities. The first pair in \eqref{eq:characterpairs} similarly gives
\[
 P_x(4,6)+2P_y(4,6)=0.
\]

Define $Q(\tau)=P(4\tau,8\tau-2)$. Since $P(x,z-2)$ has no terms of total degree zero or one,
\[
 Q(\tau)=\sum_{m\geq2}a_m\tau^m,\qquad 4^m\mid a_m.
\]
The equality $P_x(4,6)+2P_y(4,6)=0$ gives $Q'(1)=0$. The resulting virtual class has rank $r_{\mathrm{out}}=Q(1)$, so
\[
 2r_{\mathrm{out}}=\sum_{m\geq2}(2-m)a_m.
\]
The $m=2$ term vanishes; each term with $m\geq3$ is divisible by $64$. Therefore $32\mid r_{\mathrm{out}}$.

Finally, Proposition~\ref{prop:tensors} with $r=4$ gives $6E^2-E^3=8W$, of rank $32$. Thus this expression has the least positive rank allowed by Theorem~\ref{thm:nonlinearintro}. For the corresponding cellular step, take $A=6E^{\otimes2}$ and $D=E^{\otimes3}$. Their ranks are $96$ and $64$, and their multiplicity difference is $16$ in each weight. Lemmas~\ref{lem:embedding} and \ref{lem:cancellation} realize the quotient and extend it across the five- and six-dimensional cells.
\end{proof}

The same lower bound applies to symmetric-power, Schur, and gamma expressions that are integral polynomials in the exterior powers of $E$. Duals give the same classes, since $[E^*]=[E]$ on these boundaries.

\section{The stable regular index}\label{sec:ahss}

Throughout this section $X$ is a finite complex of finite groups, and $L_o,L_e$ are the least common multiples in \eqref{eq:invariants}. We use Pardon's $K$-theory of finite-rank bundles and pairs. Homotopy invariance, excision, and the long exact sequence are established in \cite[Corollary 6.4 and Propositions 6.13, 6.16, 6.21]{Pardon}. For finite complexes of groups and their subcomplexes, multiplication by the Bott element is an isomorphism by \cite[Proposition 6.24]{Pardon}. We use its periodic extension, which agrees with the bundle theory in nonpositive degrees by \cite[Corollary 6.26]{Pardon}.

The fibre representations of a bundle define a class on the left edge of the cellular spectral sequence. A compatible fibre family is realized by a virtual bundle precisely when it is a permanent cycle. Applied to the normalized regular family, this criterion gives an invariant of $X$ that can be smaller than the multiplier in Section~\ref{sec:construction}.

\begin{proposition}\label{prop:ss}
The skeletal filtration gives a convergent spectral sequence
\begin{equation}\label{eq:ss}
 E_1^{p,q}=K^{p+q}(X^p,X^{p-1})
 \cong\bigoplus_{\dim\sigma=p}K^q_{G_\sigma}(\mathrm{pt})
 \Longrightarrow K^{p+q}(X).
\end{equation}
Its even rows, with their $d_1$ differential, are the cellular cochain complex with coefficient system $\sigma\mapsto R(G_\sigma)$; the odd rows vanish.
\end{proposition}
\begin{proof}
The long exact sequences of consecutive skeletal pairs give the exact couple. Excision identifies a relative cell term with $K^{p+q}_{G_\sigma}(D^p,S^{p-1})$. Equivariant Bott periodicity \cite[Proposition~3.5]{Segal} and suspension identify it with $K^q_{G_\sigma}(\mathrm{pt})$. The boundary maps are the signed restriction maps on representations induced by the characteristic faces. The conjugacies in a complex of groups act trivially on representation rings, so these maps give a coefficient system. The finite filtration gives convergence with associated graded groups $E_\infty^{p,q}$.
\end{proof}

Characters identify the even cellular row with Pardon's $\chi$-integral cochains on the inertia \cite[Definition 2.7]{Pardon}. Thus the cochain complex occurs at $E_1$, and its cohomology at $E_2$. The normalized regular representations
\begin{equation}\label{eq:omega}
 \omega_x=\frac{L_o}{|G_x|}[\C[G_x]]
\end{equation}
are compatible under restriction by~\eqref{eq:restriction} and define a class in $E_2^{0,0}$.

\begin{theorem}[Integral annihilation at the edge]\label{thm:annihilation}
For every $s\geq1$,
\begin{equation}\label{eq:edgekill}
 g_s(L_e)d_{2s+1}=0\quad\text{on }E_{2s+1}^{0,0}.
\end{equation}
\end{theorem}
\begin{proof}
If $L_e=1$, every compatible fibre family is represented by a virtual trivial bundle on each component, so every left-edge differential vanishes. Assume $L_e\geq2$. There are no incoming differentials at the left edge. By the skeletal exact couple, $y\in E_{2s+1}^{0,0}$ lifts to a class $x\in K^0(X^{2s})$ with fibre family $y$. On the boundary of each characteristic $(2s+1)$-cell, equivariant Bott periodicity gives a decomposition
\[
 x|_{S^{2s}\times BJ}=W+\beta_sa,
\]
allowing $W$ and $a$ to be virtual representations. The operation $\Theta_s$ from Proposition~\ref{prop:least} gives $\Theta_sx=g_s(L_e)W$ on that boundary.

To extend $\Theta_sx$ over $X^{2s+1}$, choose a regular bundle $V$ on $X^{2s+1}$ by Theorem~\ref{thm:main}. There are finitely many cells and isotypic multiplicities. By adding $N[V]$ for sufficiently large $N$, we can make every fibre multiplicity of $\Theta_sx+N[V]$ positive and large enough for Corollary~\ref{cor:positive} on $X^{2s}$. Represent it by a bundle $A$. The added bundle extends across each characteristic disk, so its boundary reduced class is zero. The reduced boundary class of $A$ is consequently zero as well. Increase $N$, if necessary, to be in the cancellation range of Lemma~\ref{lem:cancellation}. The bundle $A$ then extends over $X^{2s+1}$. Subtracting the extended $N[V]$ gives an extension of $\Theta_sx$ in $K^0$.

Its fibre family is $g_s(L_e)y$, so $d_{2s+1}(g_s(L_e)y)=0$.
\end{proof}

\begin{definition}\label{def:index}
For finite connected $X$, the \emph{stable regular index} $I_K(X)$ is the positive generator of
\[
 \{a\in\Z:a\omega\in\operatorname{im}(K^0(X)\longrightarrow E_2^{0,0})\}.
\]
Equivalently, it is the least positive $a$ for which $a\omega$ is a permanent cycle. The subgroup is nonzero by Theorem~\ref{thm:main}.
\end{definition}

\begin{theorem}[Obstruction product and realization]\label{thm:index}
Let $X$ be connected with $\dim X\leq n$, put $a_0=1$, and define recursively
\[
 o_s=\operatorname{ord}\big(d_{2s+1}(a_{s-1}\omega)\big),
 \qquad a_s=a_{s-1}o_s,
 \qquad 1\leq s\leq t(n),
\]
where a zero class has order one. Then
\begin{equation}\label{eq:indexproduct}
 I_K(X)=\prod_{s=1}^{t(n)}o_s,\qquad o_s\mid g_s(L_e).
\end{equation}
Let $h_n=\max\{1,\ceil{(n-1)/2}\}$. Every positive multiple $a$ of $I_K(X)$ with $a\geq h_n$ is realized by an actual regular bundle of rank $L_oa$. In particular, if $I_K(X)\geq h_n$, the minimum regular rank is $L_oI_K(X)$.
\end{theorem}
\begin{proof}
On page $E_{2s+1}$, the multiples of $\omega$ surviving at the left edge are exactly $a_{s-1}\Z\omega$. The restriction of $d_{2s+1}$ to this cyclic group has kernel $a_{s-1}o_s\Z\omega$ by the definition of $o_s$. Theorem~\ref{thm:annihilation} shows that the image of $a_{s-1}\omega$ has finite order dividing $g_s(L_e)$, and hence $o_s\mid g_s(L_e)$. Since the coefficient groups in odd $q$-degree vanish, differentials of even length vanish. Finally, the dimension of the filtration excludes every differential beyond those indexed by $1\leq s\leq t(n)$. The permanent multiples are therefore precisely $a_{t(n)}\Z\omega$, which proves \eqref{eq:indexproduct}.

A permanent family $a\omega$ lifts to $x=[A]-[D]\in K^0(X)$. In an irreducible $\rho$ of a stabilizer $H$, its multiplicity is
\[
 a\frac{L_o}{|H|}\dim\rho\geq a.
\]
If $a\geq h_n$, Corollary~\ref{cor:positive} represents $x$ by an actual regular bundle. Conversely, the rank of any regular bundle is a positive multiple of $L_o$, and its normalized multiplier lies in the subgroup of Definition~\ref{def:index}. Thus no smaller rank than $L_oI_K(X)$ is possible, and it is attained when $I_K(X)\geq h_n$.
\end{proof}

For every value of $I_K(X)$, the theorem gives the upper bound
\begin{equation}\label{eq:actualrank}
 L_o I_K(X)\max\left\{1,\ceil{h_n/I_K(X)}\right\}.
\end{equation}
The index calculation allows a choice of virtual class on each skeleton. Cancelling a prescribed bundle can require a larger multiplier.

\begin{remark}[Infinite complexes]
On a countable disjoint union of points, $K^0$ formed from bounded-rank bundles is the group of bounded integer sequences, a proper subgroup of $\prod_{m\geq1}\Z$. We use the cellular spectral sequence for finite complexes; Section~\ref{sec:construction} gives a direct bundle construction on locally finite complexes.
\end{remark}

\section{Gerbes and regular-rank bounds}\label{sec:gerbes}

\subsection{Enlargement for stabilizers of order two}
Recall the data of a simplicial complex of groups \cite[Definition 2.1]{Pardon}. Each simplex $\sigma$ of a simplicial complex $Z$ has a group $G_\sigma$. A face inclusion $\sigma\subset\tau$ gives an injection
\[
 \phi_{\sigma\tau}:G_\tau\longrightarrow G_\sigma.
\]
For a strict chain $\rho\subset\sigma\subset\tau$ there is an element $g_{\rho\sigma\tau}\in G_\rho$ comparing the direct injection with the composite injection by conjugation. These elements satisfy the coherence identity on chains of four simplices. Degenerate chains have normalized identity data.

\begin{proposition}[Gerbe enlargement]\label{prop:enlargement}
Let $X=\|(Z,G)\|$ have all its groups isomorphic to $1$ or $\mu_2$, and let $M=|Z|$. There is a banded $\mu_2$-gerbe $\widehat X\to M$ and a representable map
\begin{equation}\label{eq:enlargementmap}
 j:X\longrightarrow\widehat X
\end{equation}
inducing the identity on coarse spaces. Pullback along $j$ preserves regular bundles and their ranks.
\end{proposition}

\begin{proof}
There is a unique injective homomorphism $\iota_\sigma:G_\sigma\to\mu_2$. Since $\operatorname{Aut}(\mu_2)=1$, these satisfy
\[
 \iota_\sigma\phi_{\sigma\tau}=\iota_\tau.
\]
Replace every cell group by $\mu_2$ and every face injection by the identity. On a strict triple use
\[
 c_{\rho\sigma\tau}=\iota_\rho(g_{\rho\sigma\tau})\in\mu_2.
\]
The conjugacy conditions are automatic in the abelian target. Identify $\mu_2$ with $\Z/2$ additively. The original coherence identity, transported by the compatible maps $\iota_\sigma$, becomes
\begin{equation}\label{eq:cocycle}
 c_{\rho\sigma\tau}-c_{\pi\sigma\tau}
 +c_{\pi\rho\tau}-c_{\pi\rho\sigma}=0
 \quad (\pi\subset\rho\subset\sigma\subset\tau).
\end{equation}
Thus $c$ is a simplicial two-cocycle on the barycentric subdivision of $Z$, and the new data define a complex of groups $\widehat G$.

In the open-star description of the realization, all the coset sets for $\widehat G$ are singletons. Consequently $\|(Z,\widehat G)\|$ is locally the product of an open star with $B\mu_2$. All maps on the band are the identity. It is the banded gerbe with class
\[
 \xi=[c]\in H^2(M;\Z/2).
\]
Using Haefliger's convention \cite[Section~2.5]{HaefligerExtensions}, define a morphism
over the identity of $Z$:
its homomorphism at $\sigma$ is $\iota_\sigma$, and its comparison element
on every face inclusion $\sigma\subset\tau$ is
$a_{\sigma\tau}=1\in\mu_2$.  The condition on a single face is
\[
 \operatorname{Ad}(a_{\sigma\tau})\circ\id_{\mu_2}\circ\iota_\tau
 =\iota_\sigma\circ\phi_{\sigma\tau},
\]
by the compatibility of the maps $\iota_\sigma$.  On a strict triple
$\rho\subset\sigma\subset\tau$, the second morphism axiom reduces, because
all the $a$'s and all target face maps are identities, to
\[
\iota_\rho(g_{\rho\sigma\tau})=c_{\rho\sigma\tau},
\]
by the definition of $c_{\rho\sigma\tau}$.  Thus the
data define a morphism of complexes of groups and hence induce
\eqref{eq:enlargementmap} on their realizations.

The map is representable. For example, on a star chart of $X$ with group $G_\sigma$, its pullback along an atlas of the target gerbe is the space induced from the $G_\sigma$-chart by the injection $\iota_\sigma$. Equivalently, the induced homomorphism on every stabilizer is injective. Finally, a regular $\mu_2$-representation restricts to a regular representation of either allowed subgroup, with unchanged vector-space dimension.
\end{proof}

\Needspace{10\baselineskip}
\begin{lemma}[Triangulation with full stabilizers]\label{rem:triangulation}
Let $X$ be a Hausdorff second-countable finite-dimensional smooth orbifold. Its coarse space has a locally finite triangulation $Z$, subordinate to the isotropy strata and with closed stars contained in chart footprints, together with a complex of finite groups $(Z,G)$ and an equivalence
\[
 \|(Z,G)\|\simeq X
\]
over the identity of $|Z|=|X|$. The group $G_\sigma$ is the full stabilizer on the interior of $\sigma$. If $|X|$ is compact, $Z$ is finite.
\end{lemma}
\begin{proof}
Apply the good-triangulation theorem of Moerdijk--Pronk to the effective reduction of $X$, using the footprints of an atlas of the original orbifold: \cite[arXiv version, Proposition~1.2.3 and Lemmas~1.2.2, 1.2.4]{MP}. It gives a locally finite triangulation subordinate to the isotropy strata, with closed stars contained in chart footprints. Its lifts form equivariant triangulations in the charts. In each original chart, the full stabilizers are the inverse images of the effective stabilizers, so the same triangulation is subordinate to the full isotropy strata.

Write $U_\sigma=\operatorname{St}(\sigma)$. Choose a point $x_\sigma$ in the interior of $\sigma$, a lift $\widetilde x_\sigma$, and the component $V_\sigma$ over $U_\sigma$ containing that lift. By \cite[arXiv version, Lemma~1.2.4 and Corollary~1.2.5]{MP}, $V_\sigma$ is a contractible open star containing exactly one lift of $x_\sigma$. Its setwise stabilizer is therefore the full stabilizer $G_\sigma$ of $\widetilde x_\sigma$, and
\[
 [V_\sigma/G_\sigma]\simeq X|_{U_\sigma}.
\]
Fix these equivalences and the resulting atlas maps $a_\sigma:V_\sigma\to X$.

For $\sigma\subset\tau$, the inclusion $U_\tau\subset U_\sigma$ makes
\[
 P_{\sigma\tau}=V_\tau\times_X V_\sigma\longrightarrow V_\tau
\]
a principal $G_\sigma$-bundle. It is trivial because $V_\tau$ is contractible. Choose a section, consisting of a lift $f_{\sigma\tau}:V_\tau\to V_\sigma$ and an arrow $a_\sigma f_{\sigma\tau}\Rightarrow a_\tau$ in $X$. For each $h\in G_\tau$, transporting its action through this section gives a continuous function $V_\tau\to G_\sigma$. Connectedness makes this function constant; composition of actions then gives a homomorphism $\phi_{\sigma\tau}:G_\tau\to G_\sigma$. At the point $\widetilde x_\tau$, this is the homomorphism on stabilizers induced by the inclusion $X|_{U_\tau}\hookrightarrow X|_{U_\sigma}$, so it is injective. Choose the identity section when $\sigma=\tau$.

For $\rho\subset\sigma\subset\tau$, the composite of the chosen sections and the direct section are two sections of the same principal $G_\rho$-bundle over $V_\tau$. They differ by a unique continuous function $V_\tau\to G_\rho$, hence by a constant element $g_{\rho\sigma\tau}$. The comparison uses the full groupoid arrows: the action on this principal bundle is free even for elements acting trivially on the chart. With the convention that the composite section is $g_{\rho\sigma\tau}$ times the direct section, we obtain
\[
 \phi_{\rho\sigma}\phi_{\sigma\tau}
 =\operatorname{Ad}(g_{\rho\sigma\tau})\phi_{\rho\tau}.
\]
Comparing the two ways to compose three chosen sections gives
\[
 \phi_{\pi\rho}(g_{\rho\sigma\tau})\,g_{\pi\rho\tau}
 =g_{\pi\rho\sigma}\,g_{\pi\sigma\tau}.
\]
Indeed, both sides compare the same composite arrow with the direct section. These are the complex-of-groups identities.

Over $U_\sigma$, the lifted interior of each $\tau\supseteq\sigma$ has $G_\sigma$-translates indexed by $G_\sigma/\phi_{\sigma\tau}(G_\tau)$. The chosen sections identify their face maps, and the elements $g_{\rho\sigma\tau}$ give the comparisons between successive face maps. Thus the pieces of the lifted triangulation lying over $U_\sigma$ have exactly the coset and gluing description of the open-star chart of $\|(Z,G)\|$ in \cite[Definition~2.1 and the following paragraph]{Pardon}. Their quotient is $[V_\sigma/G_\sigma]$. The chosen arrows identify these local equivalences on overlaps, and the displayed coherence identity makes them compatible on triple overlaps. They descend to the asserted equivalence over $|Z|$. Finally, a locally finite triangulation of a compact space has finitely many simplices.
\end{proof}

For stabilizers of order at most two, Proposition~\ref{prop:enlargement} therefore gives a gerbe over the coarse space of a smooth orbifold. If $X$ is compact, this triangulation of its coarse space is a finite complex, so Section~\ref{subsec:coarsefive} applies.

\subsection{Weights and realization}
Let $M$ be a finite connected CW complex and let $\pi:\G\to M$ be a banded $\mu_e$-gerbe, with $e\geq2$. Its class is $\xi\in H^2(M;\Z/e)$, and write
\[
 \alpha=\beta_e(\xi)\in H^3(M;\Z),
\]
where $\beta_e$ is the integral Bockstein. We identify $\mu_e$ with the $e$th roots of unity and choose the sign convention for twisted bundles so that weight one corresponds to $\alpha$.

The band fixes the characters of the stabilizer globally. The fibrewise idempotents of $\C[\mu_e]$ therefore split every bundle as
\begin{equation}\label{eq:weights}
 E=\bigoplus_{j=0}^{e-1}E_j,
\end{equation}
where $\mu_e$ acts on $E_j$ by its $j$th character. In local trivializations of the gerbe, transition maps for $E_j$ have triple product equal to the $j$th power of the gerbe cocycle. Thus weight-$j$ bundles are exactly finite-rank $j\alpha$-twisted bundles on $M$. For the gerbe-module formulation, see \cite{LU} and \cite[Definition~5.2 and Lemma~5.3]{TXLG}. In particular, taking determinants shows that a weight-$j$ bundle of rank $k$ satisfies
\begin{equation}\label{eq:detobstruction}
 kj\alpha=0.
\end{equation}

For a torsion class $\gamma$, write $\ind(\gamma)$ for the gcd of the ranks of finite-rank $\gamma$-twisted bundles. On a finite CW complex this equals the positive generator of the image of the rank map $K^0(M;\gamma)\to\Z$; see \cite[Section 2.5, Proposition 2.21 and Lemmas 2.22--2.23]{AW2014}. We set $\ind(0)=1$. The gcd is realized by a virtual twisted bundle. The proof of Proposition~\ref{prop:gerbeindex} gives a sufficient rank condition for realization by an actual bundle.

For a gerbe over an arbitrary finite connected CW complex, define $I_K(\G)$ to be the least positive integer $a$ for which $a\C[\mu_e]$ is the fibre representation of a virtual bundle at every point. When $\G$ has a finite complex-of-groups presentation, the edge homomorphism of Proposition~\ref{prop:ss} identifies this definition with Definition~\ref{def:index}.

\begin{proposition}\label{prop:gerbeindex}
For the gerbe $\G$,
\begin{equation}\label{eq:gerbeindex}
 I_K(\G)=\lcm_{0\leq j<e}\ind(j\alpha).
\end{equation}
If a positive integer $k$ is divisible by each index on the right and $2k\geq n-1$, where $\dim M\leq n$, then $\G$ has a regular bundle of rank $ek$.
\end{proposition}
\begin{proof}
A virtual bundle is regular of rank $ek$ precisely when all its weight summands have rank $k$. By \eqref{eq:weights}, such virtual classes exist precisely when each $\ind(j\alpha)$ divides $k$. This proves \eqref{eq:gerbeindex}.

The proof of Lemma~\ref{lem:embedding} also applies to twisted bundles on $M$. Pullback to a characteristic disk trivializes the gerbe. A monomorphism between two bundles of the same weight then has the ordinary complex Stiefel manifold as fibre; the scalar twisting cancels in the Hom bundle. Relative extension and cellular clutching therefore prove Lemma~\ref{lem:embedding} for each weight, with dimension that of $M$. A virtual weight-$j$ class of rank $k$ is represented by an actual twisted bundle when $2k\geq n-1$. Taking their direct sum proves the assertion.
\end{proof}

\subsection{The period--index bound}
We use the following theorem of Antieau--Williams \cite[Theorem A]{AW2021}: if a finite CW complex has dimension at most $2d$ and a torsion Brauer class $\gamma$ has period $m$, then
\begin{equation}\label{eq:awbound}
 \ind(\gamma)\mid m^{d-1}\prod_{p\mid m}p^{v_p((d-1)!)}.
\end{equation}

\begin{theorem}[Regular ranks on cyclic gerbes]\label{thm:gerbebound}
If $\dim M\leq n$, the gerbe $\G$ admits an actual regular bundle of rank
\begin{equation}\label{eq:gerbebound}
 e^{t+1}(t!)_e,\qquad t=t(n).
\end{equation}
\end{theorem}
\begin{proof}
Set $D=e^t(t!)_e$. The period of each $j\alpha$ divides $e$. Apply \eqref{eq:awbound} with $d=t+1$, so $n\leq2d$; the cases $n=0,1$ have $\alpha=0$ and the same conclusion. At each prime the bound for a divisor of $e$ divides the bound for $e$, hence $\ind(j\alpha)\mid D$ for every $j$.

Moreover, $D\geq2^t\geq t+1\geq\max\{1,\ceil{(n-1)/2}\}$. Proposition~\ref{prop:gerbeindex} supplies an actual bundle of rank $D$ in every weight. Their direct sum is regular of rank $eD$.
\end{proof}

For the universal class of period $p^a$, Antieau--Williams show that $p^{a+v_p(s)}$ annihilates the image of $d_{2s+1}$ on the rank term \cite{AW2021}. Formula~\eqref{eq:valuations} gives $g_s(p^a)=p^{a+v_p(s)}$ for odd $p$ and for $p=2$ with $a\geq2$. For $p=2$ and $a=1$, an extra factor of two occurs precisely when $s$ is even. Thus, for $2\parallel e$, the gerbe bound removes exactly the factor $2^{\floor{t/2}}$ in the universal operation construction. At order two, gerbe enlargement extends this improvement to all the orbispaces in Theorem~\ref{thm:order_two}.

\subsection{The order-two bound}
\begin{proof}[Proof of the upper bound in Theorem~\ref{thm:order_two}]
First let $X=\|(Z,G)\|$ be a possibly infinite complex of groups of dimension at most $n$, with groups $1$ and $\mu_2$. Proposition~\ref{prop:enlargement} gives a representable map to a banded $\mu_2$-gerbe over $|Z|$. Its class is represented by a map
\[
 |Z|\longrightarrow K(\Z/2,2).
\]
Choose a model of this Eilenberg--Mac Lane space with finitely many cells in every dimension. For example, the simplicial abelian group associated to $\Z/2$ in chain degree two has finitely many simplices in every degree. Its $n$-skeleton $K_n$ is a finite CW complex, so the finite-base index theorem applies there even when $Z$ has infinitely many cells. By cellular approximation the classifying map factors, up to homotopy, through $K_n$.

Theorem~\ref{thm:gerbebound} supplies a regular bundle of rank $2^{t+1+v_2(t!)}$ on the universal gerbe restricted to $K_n$. Pull it back to the enlarged gerbe and then to $X$.

For a general paracompact orbispace, apply Pardon's representable approximation \cite[Proposition 4.9]{Pardon} to obtain a map to a complex of groups of dimension at most $n$ and order at most two. Pull back the bundle just constructed. For a smooth orbifold, the smoothing and frame argument used in Theorem~\ref{thm:main} applies with the same rank.
\end{proof}

Thus
\[
 \frac{\Rsharp(n;2,2)}{R_2(n)}=2^{\floor{t/2}}.
\]
Legendre's formula also gives $R_2(n)=2^{2t+1-s_2(t)}$, where $s_2(t)$ counts the ones in the binary expansion of $t$.

\begin{table}[ht]
\centering
\caption{Uniform regular ranks $R_2(n)$ in dimension at most $n$ for stabilizers of order at most two (Theorem~\ref{thm:order_two}).}\label{tab:order_two}
\begin{tabular}{@{}crrrrrr@{}}
\toprule
$n$ & $0$--$2$ & $3$--$4$ & $5$--$6$ & $7$--$8$ & $9$--$10$ & $11$--$12$\\
\midrule
$R_2(n)$ & $2$ & $4$ & $16$ & $32$ & $256$ & $512$\\
\bottomrule
\end{tabular}
\end{table}

\subsection{Sharpness in dimension five}
For a class $\alpha$ of period two on a finite complex of dimension at most six, the twisted Atiyah--Hirzebruch spectral sequence has $d_3(1)=\alpha$ up to sign, and
\begin{equation}\label{eq:awfive}
 \ind(\alpha)=2\operatorname{ord}(d_5(2)),\qquad
 \operatorname{ord}(d_5(2))\mid4.
\end{equation}
The class $d_5(2)$ lies in $H^5(M;\Z)/(\alpha\smile H^2(M;\Z))$. The differential and its relation to the index are computed in \cite[Definition 2.4, Theorems 2.5 and 5.2, and Theorem A]{AW6}. Equation~\eqref{eq:awfive} computes the permanent rank subgroup in the twisted spectral sequence. By Proposition~\ref{prop:gerbeindex}, this index determines the virtual regular ranks of the gerbe.

\begin{theorem}[Sharp five-dimensional gerbe bound]\label{thm:five}
Among banded $\mu_2$-gerbes over finite connected CW complexes of dimension at most five, the largest possible minimum regular rank is $16$. More precisely, every such gerbe has minimum regular rank $2\ind(\alpha)$, and the value $16$ occurs on a finite five-complex.
\end{theorem}
\begin{proof}
If $\alpha=0$, a weight-one line bundle exists by the twisted-bundle description, so rank two is attained. If $\alpha\ne0$, its index is an even integer at least two. Equation \eqref{eq:gerbeindex} gives $I_K(\G)=\ind(\alpha)$, and the inequality $2\ind(\alpha)\geq4\geq\dim M-1$ permits actual realization. Conversely the rank of the odd summand of a regular bundle is divisible by $\ind(\alpha)$. Thus the minimum is $2\ind(\alpha)$. Equation \eqref{eq:awfive} gives the upper bound $16$.

To attain the bound, take $X=K(\Z/2,2)^{(6)}$ in a finite-type model and its universal period-two class $\alpha_X$. This is the example in \cite[Section 5, proof of Theorem B]{AW6}; the proof of Theorem 5.2 there computes $d_5(2)$ as a generator of $H^5(K(\Z/2,2);\Z)\cong\Z/4$, so $\ind(\alpha_X)=8$. Here $H^2(X;\Z)=0$. Put $Y=X^{(5)}$. The relative cellular cochain complex $C^*(X,Y;\Z)$ is concentrated in degree six. The long exact sequence of the pair therefore shows that restriction is an isomorphism in degrees two and three and is injective in degree five. Restriction also induces an injection
\[
 \frac{H^5(X;\Z)}{\alpha_X\smile H^2(X;\Z)}
 \longrightarrow
 \frac{H^5(Y;\Z)}{\alpha_Y\smile H^2(Y;\Z)}.
\]
Indeed, every element of $H^2(Y;\Z)$ lifts to $X$, and injectivity in degree five applies after subtracting its cup product with $\alpha_X$. The restricted class also has period two, since restriction is an isomorphism in degree three. Naturality therefore sends $d_5^{\alpha_X}(2)$ to $d_5^{\alpha_Y}(2)$, and injectivity on the quotient preserves its order four. Thus $\alpha_Y$ has period two and index eight by \eqref{eq:awfive}. Choose $\xi_Y\in H^2(Y;\Z/2)$ with $\beta_2(\xi_Y)=\alpha_Y$, possible by the Bockstein exact sequence. The associated gerbe has minimum regular rank $16$.
\end{proof}

Theorem~\ref{thm:nonlinearintro} concerns cancellation of prescribed boundary data. In Theorem~\ref{thm:order_two}, the choice of virtual classes on an enlarged gerbe gives the smaller rank $16$. Section~\ref{sec:varying} proves that this bound is sharp even with a dense free locus.

\begin{proposition}[Closed five-dimensional bases]\label{prop:closedfive}
Let $\G$ be a banded $\mu_2$-gerbe over a connected closed smooth five-manifold $M$. Its minimum regular rank is at most four if $M$ is orientable, and at most eight in general.
\end{proposition}
\begin{proof}
If $\alpha=0$, rank two suffices. Otherwise, use \eqref{eq:awfive}. For orientable $M$, $H^5(M;\Z)=\Z$. Every product $\alpha\smile b$ is torsion, hence zero in this group. The finite-order class $d_5(2)$ must also vanish. Thus $\ind(\alpha)=2$ and rank four suffices.

For nonorientable $M$, integral Poincar\'e duality with the orientation local system gives $H^5(M;\Z)=\Z/2$. The target quotient of $d_5(2)$ has exponent at most two, so $\ind(\alpha)\mid4$. Proposition~\ref{prop:gerbeindex} realizes the resulting bound in actual bundles. Smoothing preserves their stabilizer representations.
\end{proof}

A smooth orbifold with all full stabilizers equal to $\mu_2$ is a banded gerbe over a manifold: in a centered chart every nearby stabilizer has order two and hence is the entire chart group, so that group acts trivially. Its band is constant because $\operatorname{Aut}(\mu_2)$ is trivial. Gerbe enlargement permits the coarse-cohomology argument to apply to varying isotropy as well; we give it in Section~\ref{subsec:coarsefive}. We leave open whether the bound eight is attained by a banded $\mu_2$-gerbe over a closed nonorientable five-manifold.

\section{Varying isotropy in dimension five}\label{sec:varying}

Throughout this section every cell stabilizer is $1$ or $\mu_2$. Pullback from the enlarged gerbe gives a regular bundle on $X$ of the same rank. We now construct examples in which pullback preserves the minimum regular rank and the resulting orbispace has a dense free locus. Write $\smin(X)$ for the minimum positive constant rank of a regular bundle on $X$.

\subsection{Rank and odd multiplicity}
Let $X=\|(Z,G)\|$ be finite, with coarse space $M$. The simplices with group $\mu_2$ form a subcomplex $S\subset M$: if a simplex has nontrivial stabilizer, the injectivity of all face maps forces every face to have nontrivial stabilizer as well. Write $X_S$ for the suborbispace over $S$.

\begin{lemma}[Rank and odd multiplicity]\label{lem:cochains}
There is a natural isomorphism of cellular cochain complexes
\begin{equation}\label{eq:cochains}
 C^*(X;R)\cong C^*(M;\Z)\oplus C^*(S;\Z),
\end{equation}
where the left-hand side has representation-ring coefficients. At a $\mu_2$-cell it is given by
\[
 a+b\eps\longmapsto(a+b,b),
\]
and at a trivial-group cell it is the rank map.
\end{lemma}

\begin{proof}
The map $\Z\{1,\eps\}\to\Z^2$ has determinant one. Restriction between two nontrivial cell groups preserves both coordinates, and restriction to the trivial subgroup keeps the first coordinate and discards the second. These are precisely the incidence maps for the two complexes on the right of~\eqref{eq:cochains}.
\end{proof}

By Proposition~\ref{prop:ss}, the even rows of the second page are therefore
\begin{equation}\label{eq:mixedE2}
 E_2^{p,q}(X)=H^p(M;\Z)\oplus H^p(S;\Z)\qquad(q\text{ even}),
\end{equation}
and the odd rows vanish. Formula~\eqref{eq:cochains} splits the coefficient complex; higher differentials can mix the two summands.

Assume $M$ is connected and $S\ne\varnothing$. In these coordinates the normalized regular family of \eqref{eq:omega} is
\[
 \omega=(2,1_S),
\]
and its stable regular index is the integer from Definition~\ref{def:index}.

For a gerbe over $M$, the weight decomposition~\eqref{eq:weights} splits the cellular spectral sequence of Proposition~\ref{prop:ss}. Indeed, the weight projections commute with restriction, the relative bundle constructions, suspension, and multiplication by the untwisted Bott class. They therefore split the long exact sequences of skeletal pairs and their exact couple. In the weight-one summand, $E_\infty^{0,0}=\ind(\alpha)\Z$: the permanent ranks are exactly the ranks of virtual weight-one bundles.

Use \emph{weight coordinates} $(u,v)$ for these two summands. On the gerbe, they correspond to rank and odd multiplicity $(u+v,v)$ in~\eqref{eq:cochains}. If $j:X\to\G$ is the enlargement map, its map on even rows of the second page is
\begin{equation}\label{eq:jmap}
 (u,v)\longmapsto(u+v,i^*v),\qquad i:S\hookrightarrow M.
\end{equation}
In particular the regular family $(1,1)$ of the gerbe maps to $\omega$.

\subsection{Removing the stabilizers on the five-simplices}

Start with a finite five-dimensional simplicial complex $M$ and a banded $\mu_2$-gerbe $\G$ on it. Choose a two-cocycle on the face poset representing its class. This gives a complex-of-groups presentation on $M$ with group $\mu_2$ on every simplex.

Define $X$ by keeping these groups on every simplex of dimension at most four and replacing the group on each five-simplex by $1$. Retain all twisting elements. These data still form a complex of groups. Indeed, the smallest simplex in a strict triple has dimension at most three, so its group has not been changed. The conjugacy and coherence identities remain valid, and the injections from a newly trivial group are unique. The identity maps on unchanged groups and the inclusions $1\to\mu_2$ give
\begin{equation}\label{eq:topmap}
 j:X\longrightarrow\G,
 \qquad S=M^{(4)}.
\end{equation}
This is representable and is an equivalence over the four-skeleton.

\begin{theorem}[Preservation of the index]\label{thm:transfer}
In this construction suppose $M$ is connected and
\[
 H^2(M;\Z)=0.
\]
If $\xi$ is the gerbe class and $\alpha=\beta_2(\xi)$, then
\begin{equation}\label{eq:transfer}
 I_K(X)=\ind(\alpha),\qquad
 \smin(X)=2\ind(\alpha).
\end{equation}
If $\alpha\ne0$, then $\smin(X_S)=4$.
\end{theorem}

\begin{proof}
If $\alpha=0$, a regular rank-two bundle pulls back from $\G$, proving the assertion. Suppose $\alpha\ne0$. We use the cellular spectral sequence of Proposition~\ref{prop:ss} for both $X$ and $\G$, and write $d_r^{\G}$ for the differentials of the gerbe. In weight coordinates put
\[
 d_3^{\G}(0,1)=(0,\eta),\qquad \eta\in H^3(M;\Z).
\]
The family $(1,0)$ is permanent, since a trivial line bundle realizes it. We claim that $\eta\ne0$. The restriction of $\alpha$ to $M^{(3)}$ is nonzero, since cellular cohomology makes $H^3(M;\Z)\to H^3(M^{(3)};\Z)$ injective. If $\eta=0$, then $(0,1)$ would be a permanent cycle on this three-skeleton and hence the fibre family of a virtual weight-one bundle $[E]-[D]$ of rank one. Its determinant $\det(E)\otimes\det(D)^{-1}$ would be an actual weight-one line bundle. Equation~\eqref{eq:detobstruction} would force $\alpha|_{M^{(3)}}=0$, a contradiction. Theorem~\ref{thm:annihilation}, with $g_1(2)=2$, now shows that $\eta$ has order exactly two.

Thus the surviving weight-one ranks on the fifth page are $2\Z$. There is no incoming third differential at bidegree $(5,-4)$, because $H^2(M;\Z)=0$, and no outgoing differential from that bidegree in dimension five. Define
\[
 d_5^{\G}(2,2)=(0,\gamma),\qquad
 \gamma\in H^5(M;\Z),\qquad \nu=\ord(\gamma).
\]
The weight-zero component again vanishes because trivial bundles realize it. The order $\nu$ is finite because the permanent weight-one rank subgroup is nonzero. The fifth differential is the last possible edge differential, so that subgroup is $2\nu\Z$. Its bundle-level description above gives
\begin{equation}\label{eq:indexq}
 \ind(\alpha)=2\nu,
 \qquad \nu\mid4,
\end{equation}
where the divisibility follows from the index bound in \cite[Theorem A]{AW6}.

Since $S=M^{(4)}$, the relative cellular cochain complex $C^*(M,S;\Z)$ is concentrated in degree five. Thus the restriction maps $H^j(M;\Z)\to H^j(S;\Z)$ are isomorphisms for $j\leq3$. Naturality and~\eqref{eq:jmap} give $d_3^X(\omega)=(\eta,\eta|_S)$, which still has order two. Since $H^2(M;\Z)=H^2(S;\Z)=0$, the incoming $d_3$ images vanish on both sides. Moreover, $H^5(S;\Z)=0$, so the fifth-page map on the obstruction target is
\begin{equation}\label{eq:sumtarget}
 H^5(M;\Z)\oplus H^5(M;\Z)\longrightarrow H^5(M;\Z),
 \qquad (u,v)\longmapsto u+v.
\end{equation}
Naturality sends $(0,\gamma)$ to $\gamma$. Since the target map in \eqref{eq:sumtarget} restricts to the identity on the second summand, the order remains $\nu$. The first obstruction forces $k$ to be even, and the fifth obstruction to $(k/2)(2\omega)$ vanishes exactly when $\nu\mid k/2$. Hence $k\omega$ is a permanent cycle precisely when $2\nu\mid k$. There are no further differentials in dimension five, so $I_K(X)=2\nu=\ind(\alpha)$.

Every actual regular rank is a multiple of $2I_K(X)$. Theorem~\ref{thm:five} supplies a regular bundle on $\G$ of rank $2\ind(\alpha)$; pullback attains this rank on $X$.

Finally $X_S$ is the gerbe restricted to a four-dimensional base. The class $\eta|_S$ has order two, and the third differential is its only possible edge differential. Hence its permanent weight-one ranks are $2\Z$, so $\ind(\alpha|_S)=2$. Proposition~\ref{prop:gerbeindex} realizes rank two in each weight. Its minimum regular rank is therefore four.
\end{proof}

\begin{remark}\label{rem:mixing}
The order-four class in the example below survives in the \emph{coarse rank coordinate} of~\eqref{eq:cochains}, because the map on the target is the sum~\eqref{eq:sumtarget}. Its restriction to the nonfree locus is zero. Thus attaching the free five-cells raises the minimum rank beyond that detected by the gerbe on $S$.
\end{remark}

\subsection{Attainment with a dense free locus}\label{subsec:mixedsharp}
\begin{lemma}\label{lem:seed}
There is a finite connected pure five-dimensional simplicial complex $M$ and a class $\xi\in H^2(M;\Z/2)$ such that
\begin{equation}\label{eq:seed}
 H^2(M;\Z)=0,\qquad \per(\beta_2\xi)=2,\qquad\ind(\beta_2\xi)=8.
\end{equation}
Here pure means that every simplex is contained in a five-simplex.
\end{lemma}
\begin{proof}
The five-skeleton $Y$ used in the proof of Theorem~\ref{thm:five} has $H^2(Y;\Z)=0$ and a period-two class of index eight. Hatcher's simplicial replacement theorem \cite[Theorem 2C.5]{HatcherAT} gives a homotopy-equivalent finite simplicial complex of the same dimension. Transport the gerbe class across that equivalence.

If the complex is not pure, attach a five-simplex along each maximal simplex of dimension less than five, using new vertices for each attachment. A simplex strongly deformation retracts onto any chosen nonempty face, fixing that face. These finitely many attachments preserve the homotopy type and make the complex pure. The resulting $M$ has the required properties.
\end{proof}

\begin{proof}[Proof of the sharpness assertion in Theorem~\ref{thm:order_two}]
Take the gerbe of the class $\xi$ from Lemma~\ref{lem:seed}, and perform the top-cell replacement~\eqref{eq:topmap}. Theorem~\ref{thm:transfer} gives
\[
 \smin(X)=16,\qquad \smin(X_S)=4.
\]
The free locus is exactly the union of the open five-simplices. It is open and dense because $M$ is pure. The four-skeleton is nonempty and has stabilizer $\mu_2$ throughout. Thus $X$ is a finite connected complex of groups with the asserted ranks and stabilizer loci.
\end{proof}

\begin{corollary}[Stable and actual ranks in dimension five]\label{cor:actualindex}
For any finite connected complex of groups of dimension at most five with stabilizers of order at most two and a nonempty nonfree locus,
\[
 \smin(X)=2I_K(X),\qquad
 I_K(X)\mid8.
\]
In particular the only possible minimum ranks are $2,4,8,16$.
\end{corollary}

\begin{proof}
The rank-sixteen upper bound shows $I_K(X)\mid8$. If $I_K(X)\geq2$, Corollary~\ref{cor:positive} represents a virtual class with fibre family $I_K(X)\omega$ by an actual bundle. At a $\mu_2$-cell both multiplicities are $I_K(X)$, and at a free cell the rank is twice that value.

If $I_K(X)=1$, choose a virtual regular class $[E]-[D]$ with fibre family $\omega$. Its determinant line
\[
 L=\det(E)\otimes\det(D)^{-1}
\]
has sign character at every $\mu_2$-point, since the virtual odd multiplicity is one. Then $\mathbf1\oplus L$ is an actual regular bundle of rank two. The lower bound follows in either case from the definition of $I_K(X)$.
\end{proof}

An ordinary space has minimum regular rank one, which explains the nonempty-isotropy hypothesis.

\subsection{Coarse top cohomology and the oriented bound}\label{subsec:coarsefive}
The $2$-primary torsion in $H^5(M;\Z)$ gives a sharper rank bound. For a finite complex $M$, let $e_5(M)$ be the exponent of this subgroup, with exponent one for the zero subgroup.

\begin{lemma}[Top cohomology of an oriented coarse space]\label{lem:coarsetop}
Let $X$ be a connected closed oriented smooth $n$-orbifold without boundary, with $n\geq2$, and let $M=|X|$. A good triangulation subordinate to the isotropy stratification makes $M$ a connected oriented $n$-dimensional pseudomanifold. In particular,
\[
 H^n(M;\Z)\cong\Z.
\]
\end{lemma}
\begin{proof}
Pass to the effective reduction, which does not change the coarse space or its orientation. In an oriented orbifold every local stabilizer acts by orientation-preserving linear transformations. A nonidentity element therefore cannot have a fixed hyperplane: a finite-order real transformation that is the identity on a hyperplane acts by $-1$ on its normal line and reverses orientation. Hence the singular locus of the effective reduction has codimension at least two.

Choose a good triangulation $Z$ subordinate to the isotropy stratification, as in Lemma~\ref{rem:triangulation}. Its lifts triangulate the local $n$-dimensional charts, so it is pure of dimension $n$. Every open $(n-1)$-simplex lies in the regular stratum and is incident to exactly two $n$-simplices. The orbifold orientation gives coherent orientations of the top simplices, with opposite induced orientations along every common face.

The regular stratum is connected. Indeed, in a centered effective linear chart $B/G$, its preimage is
\[
 B\setminus\bigcup_{1\ne g\in G}B^g.
\]
This is the complement in a ball of finitely many linear subspaces of codimension at least two. It is dense and path connected: two points can be joined by two line segments through a generic intermediate point, both avoiding those subspaces. Its quotient is thus also path connected. The connected space $M$ is locally path connected. Cover a path in $M$ by a finite chain of these chart neighbourhoods; each consecutive overlap contains a regular point by density. Paths through their regular parts then join any two regular points. Within this regular manifold, a path between interiors of top simplices can be perturbed to avoid the $(n-2)$-skeleton and cross the $(n-1)$-faces one at a time. Thus the top-simplex adjacency graph of $Z$ is connected, and $M$ is an oriented pseudomanifold.

In the top cellular cochain group, the coboundary of each oriented $(n-1)$-simplex is the difference of its two incident top-simplex generators. Since the adjacency graph is connected, the quotient by these differences is freely generated by the common class of the top simplices. Therefore the cokernel is $\Z$.
\end{proof}

\begin{proposition}\label{prop:coarsebound}
Let $X$ be a finite complex of groups of dimension at most five, with groups $1$ and $\mu_2$ and coarse space $M$. Then it has a regular bundle of rank
\begin{equation}\label{eq:coarserank}
 4\gcd(4,e_5(M)).
\end{equation}
If the Bockstein of the enlarged gerbe class vanishes, rank two suffices. In particular, if $H^5(M;\Z)$ has no $2$-primary torsion, rank four suffices.
\end{proposition}

\begin{proof}
Enlarge $X$ over $M$ and let $\alpha$ be the Bockstein of its gerbe class. The vanishing case follows from Theorem~\ref{thm:five}. Suppose $\per(\alpha)=2$. The computation of \cite[Theorem A]{AW6} gives
\begin{equation}\label{eq:Q}
 \ind(\alpha)=2\nu,\qquad
 \nu=\ord\bigl(\widetilde Q(\alpha)\bigr)\mid4,
 \qquad
 \widetilde Q(\alpha)\in
 \frac{H^5(M;\Z)}{\alpha\smile H^2(M;\Z)}.
\end{equation}
The image $\alpha\smile H^2(M;\Z)$ is contained in the $2$-primary torsion subgroup and is annihilated by two. The torsion subgroup of $H^5(M;\Z)/(\alpha\smile H^2(M;\Z))$ is therefore a quotient of the torsion subgroup of $H^5(M;\Z)$. Its $2$-primary exponent divides $e_5(M)$, so $\nu\mid\gcd(4,e_5(M))$.

An actual regular bundle of rank $4\nu$ on the enlarged gerbe exists by Theorem~\ref{thm:five}. Taking the appropriate direct sum and pulling back gives exactly~\eqref{eq:coarserank}.
\end{proof}

\begin{proof}[Proof of the upper bound in Theorem~\ref{thm:oriented_five}]
Let $M$ be the coarse space of a connected closed oriented smooth five-orbifold with stabilizers of order at most two. Orientation means an orientation of the tangent orbibundle, so every local stabilizer preserves it. Lemma~\ref{lem:coarsetop} gives
\begin{equation}\label{eq:topcohom}H^5(M;\Z)\cong\Z.\end{equation}
Use the complex-of-groups presentation over this same coarse triangulation from Lemma~\ref{rem:triangulation}. Proposition~\ref{prop:coarsebound} gives a continuous regular bundle of rank four. The smoothing argument in the proof of Theorem~\ref{thm:main}, using \cite[Proposition 4.1]{BL}, preserves its rank and every stabilizer representation. For a disconnected compact orbifold apply the argument on each component and use rank four on all components.
\end{proof}

\Needspace{11\baselineskip}
\section{Sharp examples and dimension dependence}\label{sec:examples}

\subsection{A smooth three-dimensional example}
Let
\begin{equation}\label{eq:rootgerbe}
 \G_2=[S^3/U(1)],\qquad z\cdot(v_1,v_2)=(z^2v_1,z^2v_2).
\end{equation}
This locally free action presents a smooth ineffective two-orbifold. Its coarse space is $S^2$, all stabilizers are $\mu_2$, and it is the nontrivial banded $\mu_2$-gerbe whose class is the mod-two reduction of a generator of $H^2(S^2;\Z)$. Let $\eta$ denote the line bundle associated with a weight-one character, choosing the sign of the Hopf line $H$ so that
\begin{equation}\label{eq:eta}
 \eta^{\otimes2}\cong\pi^*H.
\end{equation}

\begin{lemma}\label{lem:pic}
The Picard group of $\G_2$ is infinite cyclic, generated by $\eta$. Its odd line bundles are precisely $\eta^{\otimes(2a+1)}$, $a\in\Z$.
\end{lemma}
\begin{proof}
Equivariant complex line bundles are classified by the second integral cohomology of the Borel space. In the fibration
\[
 S^3\longrightarrow EU(1)\times_{U(1)}S^3\longrightarrow BU(1),
\]
the Serre spectral sequence has only the fibre rows in degrees zero and three. In total degree two it gives an isomorphism from $H^2(BU(1);\Z)=\Z$ to the Borel cohomology. Its generator classifies $\eta$. Restriction to the stabilizer $\{\pm1\}$ sends $\eta$ to the sign character, proving the parity assertion. The quotient Hopf bundle is associated with the squared character, giving \eqref{eq:eta}.
\end{proof}

Complex conjugation $c:S^3\to S^3$ is equivariant for the automorphism $z\mapsto\overline z$ of $U(1)$. It induces a smooth automorphism $\widetilde c$ of $\G_2$ covering a degree $-1$ map of $S^2$. It fixes the band $\mu_2$ and satisfies $\widetilde c^*\eta\cong\eta^{-1}$. Let $T$ be its mapping torus.

\begin{theorem}\label{thm:sharpfour}
The compact smooth three-orbifold $T$ has constant stabilizer $\mu_2$ and minimum complex regular rank four. The same minimum holds for $T\times S^1$ in dimension four.
\end{theorem}
\begin{proof}
A regular rank-two bundle on $T$ would have an odd line summand. Restriction to a fibre $\G_2$ would give an odd line $E_-$ with $\widetilde c^*E_-\cong E_-$, because the bundle glues around the mapping torus and the band is fixed. By Lemma~\ref{lem:pic}, write $E_-=\eta^{2a+1}$. The isomorphism would imply $2a+1=-(2a+1)$ in $\Pic(\G_2)=\Z$, which is impossible.

Rank four is attained directly. On $\G_2$, take
\[
 \underline\C^{\,2}\oplus\eta\oplus\eta^{-1}.
\]
It has two copies of each stabilizer character. The identity on the trivial summand and interchange of the two odd lines give the required gluing under $\widetilde c$. Thus it descends to $T$. Pullback gives rank four on $T\times S^1$, while restriction to $T\times\{1\}$ rules out rank two there.
\end{proof}

The theorem proves that the uniform ranks $R(3,2)=R(4,2)=4$ of Part~I are optimal. It also identifies the first stable obstruction. The coarse space $M$ of $T$ is the mapping torus of the degree $-1$ map of $S^2$. The Wang sequence gives
\[
 H^2(M;\Z)=0,\qquad H^3(M;\Z)=\Z/2.
\]
With mod-two coefficients, restriction $H^2(M;\Z/2)\to H^2(S^2;\Z/2)$ is an isomorphism. The class of the gerbe is its nonzero element, whose integral Bockstein is nonzero. Thus $d_3$ in the twisted spectral sequence forces the rank of every odd-weight bundle to be even.

\subsection{A closed oriented five-dimensional example}
Write
\[
 Y=S^3\times S^2\times S^1,
 \quad S^3\subset\C^2,
 \quad S^2=\{(w,h)\in\C\oplus\R:|w|^2+h^2=1\}.
\]
Let $O(2)=U(1)\rtimes\langle r\rangle$, with $r^2=1$ and $rzr=z^{-1}$. With $\theta\in\R/2\pi\Z$, define an action by
\begin{align}
 z\cdot(v,w,h,\theta)&=(z^2v,zw,h,\theta),\label{eq:rotation}\\
 r\cdot(v,w,h,\theta)&=(\bar v,\bar w,h,\theta+\pi).
 \label{eq:reflection}
\end{align}
The formulas satisfy the defining relations of $O(2)$. Put
\[
 \mathcal O=[Y/O(2)].
\]

\begin{proposition}\label{prop:Oexample}
The quotient $\mathcal O$ is a connected closed oriented effective smooth five-orbifold. Its stabilizers are $1$ and $\mu_2$, its free locus is dense, and
\[
 \smin(\mathcal O)=4.
\]
\end{proposition}

\begin{proof}
An element in the reflection component shifts $\theta$ by $\pi$, so it cannot fix a point. A rotation fixing $v\in S^3$ must satisfy $z^2=1$. If $w\ne0$, it must also satisfy $z=1$. If $w=0$, both $z=1$ and $z=-1$ fix the point. Thus the nonfree locus is given by the two poles of the $S^2$ factor, and all stabilizers are finite. The compact group action is proper and locally free, so the quotient is a smooth orbifold of dimension $6-1=5$. It is closed and connected. Its free locus $w\ne0$ is open and dense. At a point with stabilizer $\mu_2$, the nonidentity element acts as $-1$ on the two real normal directions to the pole locus, proving effectiveness.

For orientation, conjugation on $S^3\subset\C^2$ has degree $+1$, while $(w,h)\mapsto(\bar w,h)$ on $S^2$ has degree $-1$. Translation on $S^1$ preserves orientation. Thus $r$ reverses the orientation of $Y$. It also reverses the orientation of the one-dimensional orbit Lie algebra, since $rzr=z^{-1}$. The two signs cancel on
\[
 \det(TY)\otimes\det(\mathfrak{o}(2))^{-1},
\]
the orientation line of the quotient tangent bundle. Rotations preserve both factors. This gives the stated orbifold orientation.

Consider the pole $w=0,h=1$. The corresponding suborbifold is
\[
 T=[(S^3\times S^1)/O(2)].
\]
First divide by $U(1)$, which acts with weight two on $S^3$. This gives the root gerbe
\[
 \G_2=[S^3/U(1)]\longrightarrow\C P^1.
\]
The remaining involution acts by complex conjugation on $\G_2$ and by a half-turn on $S^1$. Thus $T$ is the mapping torus of Theorem~\ref{thm:sharpfour}. A regular rank-two bundle on $\mathcal O$ would restrict to one on $T$, contradicting that theorem.

Rank four has a direct construction. Let $\rho_{\R}$ be the standard real two-dimensional representation of $O(2)$ and use the complex representation
\begin{equation}\label{eq:explicitbundle}
 V=\C^2_{\mathrm{triv}}\oplus(\rho_{\R}\otimes_{\R}\C).
\end{equation}
The rotation $-1$ acts as $+1$ on the first two dimensions and as $-1$ on the last two. The bundle $[Y\times V/O(2)]\to\mathcal O$ is therefore regular of rank four at every nontrivial stabilizer, and it is regular at every free point. Since a regular rank at a $\mu_2$-point is even, this proves the minimum is four and completes Theorem~\ref{thm:oriented_five}.
\end{proof}

\subsection{A three-dimensional example for every prime}
The preceding smooth mapping-torus argument uses the fact that inversion fixes $\mu_2$. For larger primes one can instead use a finite CW base.

\begin{theorem}\label{thm:prime}
For every prime $p$, there is a banded $\mu_p$-gerbe over a finite connected three-dimensional CW complex whose minimum regular rank is $p^2$.
\end{theorem}
\begin{proof}
Let $f:S^2\to S^2$ be a cellular map of degree $1+p$, and let
\[
 M_p=(S^2\times[0,1])/\big((x,1)\sim(f(x),0)\big).
\]
This mapping torus is a finite three-dimensional CW complex. Its cohomology sequence gives
\[
 H^2(M_p;\Z)=\ker(1-(1+p):\Z\to\Z)=0,
\]
\[
 H^3(M_p;\Z)=\operatorname{coker}(1-(1+p):\Z\to\Z)=\Z/p.
\]
The Bockstein $H^2(M_p;\Z/p)\to H^3(M_p;\Z)$ is therefore an isomorphism. Let $\xi$ map to a generator $\alpha$, and form its banded $\mu_p$-gerbe.

A regular bundle of rank $pk$ has a weight-one summand of rank $k$. Its determinant forces $k\alpha=0$ by \eqref{eq:detobstruction}; hence $p\mid k$. Theorem~\ref{thm:gerbebound}, with $n=3$ and $e=p$, supplies rank $p^2$. This is the minimum.
\end{proof}

This proves sharpness for the single stabilizer group $\mu_p$, with $L_o=L_e=p$. The uniform problem allowing every group of order at most $p$ has a different bound when $p>2$.

\subsection{Unbounded regular rank with fixed cyclic isotropy}
Antieau--Williams construct finite complexes with a fixed period $p$ and unbounded topological index \cite[Section 5.2]{AW2014}. Smooth thickening transfers their obstruction to closed smooth orbifolds.

\begin{theorem}\label{thm:unbounded}
For every prime $p$ and every integer $N$, there is a connected compact smooth orbifold without boundary, with every full stabilizer equal to $\mu_p$, on which every complex regular bundle has rank greater than $N$.
\end{theorem}
\begin{proof}
Choose a finite connected CW complex $K$ with a period-$p$ class $\alpha$ for which $p\ind(\alpha)>N$. Lift $\alpha$ to $\xi\in H^2(K;\Z/p)$ by the Bockstein exact sequence. Replace $K$ by a homotopy-equivalent finite polyhedron, carrying these classes along the equivalence.

After subdivision, realize this polyhedron as a finite subcomplex of a smooth triangulation of $\R^m$, with $m$ sufficiently large. Hirsch's theorem on smooth regular neighbourhoods \cite[Theorem~1(a)]{Hirsch} supplies a compact smooth codimension-zero submanifold $U\subset\R^m$ with boundary that is a PL regular neighbourhood of $K$. Thus $U$ collapses onto $K$. Choose a homotopy inverse $h:U\to K$ to the inclusion. Double $U$ along its boundary to obtain a closed connected smooth manifold $M$. Composing the fold map $M\to U$ with $h$ gives $r:M\to K$. The inclusion $i:K\to U\to M$ satisfies $ri\simeq\id_K$.

The class $r^*\xi$ defines a smooth banded $\mu_p$-gerbe on $M$. Choose a good cover of $M$ and represent $r^*\xi$ by locally constant \v Cech cocycles. The resulting proper local quotient charts have trivial $\mu_p$ action. Thus the gerbe is a compact smooth orbifold with precisely the asserted stabilizers.

Let $E$ be a regular bundle of rank $pk$ on the gerbe. Its weight-one summand is an $r^*\alpha$-twisted bundle of rank $k$. Pullback along $i$ gives an $\alpha$-twisted bundle on $K$. Hence $\ind(\alpha)\mid k$ and $pk\geq p\ind(\alpha)>N$.
\end{proof}

\subsection{Sharpness in dimension six}
\begin{corollary}\label{cor:sixsharp}
The largest possible minimum regular rank of a connected closed oriented smooth six-orbifold with stabilizers of order at most two is $16=R_2(6)$. It is attained by a banded $\mu_2$-gerbe over a closed oriented six-manifold.
\end{corollary}
\begin{proof}
Theorem~\ref{thm:order_two} gives the upper bound $R_2(6)=16$. Crowley--Grant construct a connected closed smooth six-manifold $N$ and a class $x\in H^2(N;\Z/2)$ for which
\[
 \alpha=\beta_2(x),\qquad \per(\alpha)=2,\qquad \ind(\alpha)=8;
\]
see \cite[Theorem 1.4(2) and Proposition 5.9]{CG}. Their manifold is the total space of a two-sphere bundle over a nonorientable four-manifold. The total space $N$ is orientable by \cite[Lemma 5.4(i)]{CG}.

Represent $x$ by a locally constant \v{C}ech two-cocycle on a good cover of $N$. The resulting banded gerbe $\G\to N$ has smooth charts $U\times B\mu_2$, with $\mu_2$ acting trivially on $U$. It is a connected closed smooth six-orbifold, oriented by the pullback of an orientation of $N$.

A regular bundle of rank $2k$ on $\G$ has a weight-one summand of rank $k$. Its rank is divisible by $\ind(\alpha)=8$, so $16\mid 2k$. Conversely, Proposition~\ref{prop:gerbeindex} realizes rank eight in each weight, since $2\cdot8\geq6-1$. Their direct sum is a regular bundle of rank sixteen. Smoothing as in the proof of Theorem~\ref{thm:main} preserves these weights and their ranks.
\end{proof}

The stabilizer of this gerbe acts trivially in every chart; its free locus is empty.

\section{Further questions}\label{sec:questions}

For a fixed family of stabilizer groups, the least multiplier in \eqref{eq:universal} depends on the Adams operations on their representation rings, as Example~\ref{ex:q8} shows. One can then ask how this multiplier relates to the obstructions on spaces with those stabilizers.

\begin{question}
For which finite families of stabilizers does the minimum universal Adams multiplier give the order of an actual edge differential on a finite complex of groups?
\end{question}

On a finite five-complex with stabilizers of order at most two, the fifth differential on the surviving normalized regular family has order at most four, although $g_2(2)=8$. Indeed, $8\omega$ survives by Theorem~\ref{thm:order_two}. If the first obstruction has order two, the next order divides four. If the first obstruction vanishes, a virtual regular class on the four-skeleton has determinant an odd line. That line extends across the five-cells: its pullback to each $S^4\times BH$ is a trivial line with a fixed character of $H$, and hence extends over the characteristic disk. Its sum with a trivial line realizes $\omega$, so the next obstruction vanishes as well. The example in Section~\ref{sec:varying} attains order four.

The statement concerns the normalized regular family. When the nonfree locus is disconnected, other fibre families need not come from constant gerbe weights.

\begin{question}
What is the sharp bound among closed nonorientable smooth five-orbifolds with both stabilizer types?
\end{question}

Theorem~\ref{thm:oriented_five} answers the oriented case. Proposition~\ref{prop:coarsebound} gives a bound in terms of integral top cohomology. The example attaining sixteen in Section~\ref{sec:varying} is a finite complex of groups; attainment in the smooth nonorientable category remains to be determined.

\begin{question}
In which higher dimensions is $R_2(n)$ attained with a dense free locus, and what replaces gerbe enlargement for stabilizers $1$ and $\mu_p$ with $p$ odd?
\end{question}

Corollary~\ref{cor:sixsharp} gives a closed oriented smooth example attaining $R_2(6)$, with constant ineffective stabilizer and empty free locus. In higher dimensions, the target of the comparison in Theorem~\ref{thm:transfer} can receive further incoming differentials. For odd primes, the automorphisms of $\mu_p$ also allow band monodromy.

Appendix~\ref{sec:real} treats real regular bundles, including the first obstruction and the Stiefel connectivity estimates for the three representation types.

\appendix
\section{The first real obstruction}\label{sec:real}

The underlying real bundle of a complex regular bundle of rank $R$ is real regular of rank $2R$, since
\[
 \C[G]_{\R}\cong\R[G]\oplus\R[G].
\]
Theorems~\ref{thm:main} and \ref{thm:order_two} therefore give explicit real bounds. Theorem~\ref{thm:oriented_five} gives a real regular bundle of rank eight on every closed oriented smooth five-orbifold with stabilizers of order at most two. For real bundles, extension over the two-skeleton can already be obstructed, as the following example shows.

\begin{proposition}\label{prop:real}
The root gerbe $\G_2$ of \eqref{eq:rootgerbe} has minimum complex regular rank two and minimum real regular rank four.
\end{proposition}
\begin{proof}
The complex bundle $1\oplus\eta$ has rank two and is regular. A real regular bundle of rank two would have an odd real line summand $\ell$. Give $\ell$ a metric. Its complexification has transition functions in $\{\pm1\}$, so $(\ell\otimes_\R\C)^{\otimes2}$ is trivial. It is also an odd complex line on $\G_2$. Lemma~\ref{lem:pic} says that every odd line has nonzero odd class in the torsion-free group $\Pic(\G_2)=\Z$, which is impossible for a line whose square is trivial. Finally, $\underline\R^{\,2}\oplus\eta_\R$ has two copies of each real stabilizer character and is regular of rank four.
\end{proof}

\Needspace{11\baselineskip}
For a real irreducible representation, its commuting division algebra is $\R$, $\C$, or $\mathbb H$. Isotypic multiplicity bundles accordingly have real, complex, or quaternionic structure. Their stable sphere groups are the corresponding $KO$, $K$, or $KSp$ groups, and their injection spaces have connectivity
\[
 \begin{array}{c|c}
 \text{space, for }b>0 & \text{connectivity}\\\hline
 V_b(\R^a)&a-b-1\\
 V_b(\C^a)&2(a-b)\\
 V_b(\mathbb H^a)&4(a-b)+2.
 \end{array}
\]
These formulas follow from the same sphere fibrations as in Lemma~\ref{lem:embedding}. The coefficient groups $\widetilde{KO}^0(S^1)$ and $\widetilde{KO}^0(S^2)$ are both $\Z/2$ \cite{Atiyah}; hence a real obstruction can occur at a two-cell, as the proposition demonstrates.

\Needspace{18\baselineskip}
\section*{Acknowledgments}
This paper continues the work of Enrique Becerra's doctoral thesis at CINVESTAV, supervised by Ernesto Lupercio.

Ernesto Lupercio acknowledges support from CONAHCYT (grant CB-2017-2018-A1-S-30345), the Simons Foundation (grant SFI-MPS-T-Institutes-00007697), and the Ministry of Education and Science of the Republic of Bulgaria (grant DO1-239/10.12.2024). He thanks CINVESTAV for his sabbatical leave, and the International Center for Mathematical Sciences--Sofia at the Institute of Mathematics and Informatics, Bulgarian Academy of Sciences, and the Institute for the Mathematical Sciences of the Americas at the University of Miami for their hospitality and support.

Claude (Anthropic) and ChatGPT (OpenAI) assisted with checking calculations and arguments and improving the English exposition. The mathematical ideas are due to the authors. The authors take responsibility for the results and their presentation.

\clearpage
\bibliographystyle{amsplain}
\bibliography{regular_bundles_II}
\end{document}